\documentclass[11pt,reqno,oneside]{amsart}
\usepackage{graphicx, subfigure}
\usepackage{amssymb}
\usepackage{epstopdf}
\usepackage{amssymb}
\usepackage[a4paper, total={6in, 9.1in}]{geometry}

\usepackage{bm}
\usepackage{amsmath,amsfonts,amsthm,mathrsfs,amssymb,cite}
\usepackage[usenames]{color}

\usepackage{graphics}
\usepackage{framed}
\usepackage{appendix}

\usepackage{enumerate}
\usepackage{hyperref}
\usepackage{doi}
\usepackage{tikz}
\usepackage{xcolor}
\hypersetup{
  colorlinks,
  linkcolor={blue!80!black},
  urlcolor={blue!80!black},
  citecolor={blue!80!black}
}
\usepackage{zref-clever}
\zcsetup{cap,abbrev=false}
\usetikzlibrary{arrows.meta,calc}

\AddToHook{env/theorem/begin}{%
\zcsetup{countertype={theorem=theorem}}}
\AddToHook{env/lemma/begin}{%
\zcsetup{countertype={theorem=lemma}}}
\AddToHook{env/proposition/begin}{%
\zcsetup{countertype={theorem=proposition}}}
\AddToHook{env/corollary/begin}{%
\zcsetup{countertype={theorem=corollary}}}
\AddToHook{env/definition/begin}{%
\zcsetup{countertype={theorem=definition}}}
\AddToHook{env/assumption/begin}{%
\zcsetup{countertype={theorem=assumption}}}
\AddToHook{env/example/begin}{%
\zcsetup{countertype={theorem=example}}}
\AddToHook{env/remark/begin}{%
\zcsetup{countertype={theorem=remark}}}
\AddToHook{env/exo/begin}{%
\zcsetup{countertype={theorem=exercise}}}

\newtheorem{theorem}{Theorem}[section]
\newtheorem{lemma}[theorem]{Lemma}
\newtheorem{proposition}[theorem]{Proposition}

\theoremstyle{definition}

\newtheorem*{notation*}{Notation}
\newtheorem{assumption}[theorem]{Assumption}
\zcRefTypeSetup{assumption}{
Name-sg = Assumption ,
name-sg = assumption ,
Name-pl = Assumptions ,
name-pl = assumptions ,
}

\theoremstyle{remark}
\newtheorem{remark}[theorem]{Remark}

\newcommand{\thmref}[1]{\hyperref[{#1}]{\nameref*{#1} \zcref*{#1}}}

\newcounter{saveeqn}

\newcommand{\eqnref}[1]{(\ref {#1})}

\DeclareMathAlphabet{\mathbbold}{U}{bbm}{m}{n}

\DeclareFontFamily{U}{wncy}{}
    \DeclareFontShape{U}{wncy}{m}{n}{<->wncyr10}{}
    \DeclareSymbolFont{mcy}{U}{wncy}{m}{n}
\DeclareMathSymbol{\Sh}{\mathord}{mcy}{"58} 

\makeatletter
\DeclareRobustCommand\widecheck[1]{{\mathpalette\@widecheck{#1}}}
\def\@widecheck#1#2{%
    \setbox\z@\hbox{\m@th$#1#2$}%
    \setbox\tw@\hbox{\m@th$#1%
       \widehat{%
          \vrule\@width\z@\@height\ht\z@
          \vrule\@height\z@\@width\wd\z@}$}%
    \dp\tw@-\ht\z@
    \@tempdima\ht\z@ \advance\@tempdima2\ht\tw@ \divide\@tempdima\thr@@
    \setbox\tw@\hbox{%
       \raise\@tempdima\hbox{\scalebox{1}[-1]{\lower\@tempdima\box
\tw@}}}%
    {\ooalign{\box\tw@ \cr \box\z@}}}
\makeatother

\newcommand{\ba}{\begin{array}}
\newcommand{\ea}{\end{array}}
\newcommand{\bsa}{\begin{subarray}}
\newcommand{\esa}{\end{subarray}}
\newcommand{\dvg}{{\rm div}\,}

\newcommand{\bea}{\begin{eqnarray*}}
\newcommand{\eea}{\end{eqnarray*}}
\newcommand{\bean}{\begin{eqnarray}}
\newcommand{\eean}{\end{eqnarray}}

\newcommand{\tendvers}[2]{\underset{{#1}\rightarrow {#2}}{\longrightarrow}}

\newcommand{\setdef}[2]{\left\lbrace {#1} \; \middle| \; {#2} \right\rbrace}

\newcommand{\R}{\mathbb{R}}
\newcommand{\C}{\mathbb{C}}

\newcommand{\N}{\mathbb{N}}

\renewcommand{\S}{\mathbb{S}}

\newcommand{\Sw}{\mathcal{S}}
\newcommand{\Cont}{\mathcal{C}}

\newcommand{\Lop}{\mathscr{L}}

\newcommand{\n}{\nabla}
\newcommand{\dist}{{\rm dist}}

\renewcommand{\Re}{{\rm Re}}

\renewcommand{\a}{\alpha}
\renewcommand{\b}{\beta}
\renewcommand{\d}{\delta}
\newcommand{\e}{\varepsilon}
\newcommand{\g}{\gamma}
\newcommand{\vf}{\varphi}
\newcommand{\lm}{\lambda}

\renewcommand{\t}{\theta}

\newcommand{\p}{\partial}
\renewcommand{\O}{\Omega}
\newcommand{\ds}{\displaystyle}

\title[stability for polyhedral inclusions]{On the Stability of Inverse Conductivity Problem for Polyhedral Inclusions under a Single Measurement}

\author{Chun-Hsiang Tsou}
\address{National Central University, Taoyuan, Taiwan}
\email{chtsou@math.ncu.edu.tw, chun.hsiang.tsou@gmail.com}

\date{} 
\begin{document}

\begin{abstract}

In this paper, we study the stability of the inverse conductivity problem of determining a convex polyhedral inclusion embedded in a homogeneous isotropic medium from a single boundary measurement. The main tools in our analysis are singularity decomposition for elliptic equations in non-smooth domains, propagation of smallness, and microlocal analysis. Combining these tools, we establish a double-logarithmic stability estimate for the Hausdorff distance between inclusions in terms of the measurement error.



\end{abstract}

\thanks{This research was partially supported by National Science and Technology Council through grant number NSTC 113-2115-M-008-003-MY2.}

\maketitle

\section{Introduction}

We first present the mathematical formulation of the problem and the principal result. We then briefly situate our work within the existing literature. Before presenting the proofs, we describe the technical ingredients and the structure of the paper.

\subsection{Mathematical formulation and main result}

In this paper, we consider the classical Calder\'on problem with the following conductivity equation.

\begin{align}\label{eq:calderon}
    \left\lbrace  \ba{lcc} \dvg(\sigma \n u)=0 & {\rm in}& \O, \\ u=f & {\rm on} & \p\O, \ea \right.
\end{align}
 where $\O$ denotes a bounded Lipschitz domain in $\R^3$.

 We focus on the case where the conductivity distribution $\sigma$ is piecewise constant with a jump across an interface. More precisely, we assume that 
 \begin{align*}
        \sigma=1+(k-1)\chi_D,
 \end{align*}
where $k\in \R_+$ is a positive constant representing the conductivity contrast between the inclusion and the surrounding medium, and $\chi_D$ denotes the indicator function of the subdomain $D\Subset \O$, which takes the value $1$ inside $D$ and $0$ outside. We are particularly interested in the case where $D$ has a special geometric shape, namely a convex polyhedron. On the boundary $\p\O$, we impose a non-constant Dirichlet condition $f\in H^{1/2}(\p\O)$. Standard elliptic theory implies that the forward problem \eqref{eq:calderon} is well posed. Equivalently, this equation can be rewritten with transmission conditions as follows.

 \begin{align}\label{eq:transmission}
    \left\lbrace \ba{lcc} \Delta u=0 & {\rm in} & D\cup (\O\setminus \overline{D}) \\ u|_-=u|_+ & {\rm on} & \p D \\ k\p_\nu u|_-=\p_\nu u|_+ & {\rm on} & \p D \\ u=f & {\rm on} & \p\O. \ea \right.
 \end{align}
 Here, the subscripts $-$ and $+$ represent the limits from the interior and exterior of $D$ respectively. The normal derivative $\p_\nu$ is defined as $\p_\nu u=\n u\cdot \nu$ where $\nu$ is the unit outward normal vector on $\p D$.
 
The inverse problem considered here is to determine the unknown inclusion $D$ from a single partial boundary measurement $\p_\nu u|_{\Gamma_0}$. Here, $\Gamma_0$ is an open subset of $\p\O$ representing the accessible part of the boundary where measurements are available. The main result of this paper is a quantitative estimate of the reconstruction error for $D$ in terms of the measurement error.

Before stating the main result rigorously, we clarify the geometric framework of our analysis. Many parameters may influence the stability estimate, but we focus only on its dependence on the measurement error. We therefore introduce the following assumptions to fix the geometric setting and avoid unnecessary complications. All these parameters are treated as {\it a priori} data, meaning that they are known in advance and are not part of the stability estimate.

 \medskip

 \begin{assumption}[Geometric]\label{assumption:geometric}
    Throughout this paper, we impose the following geometric conditions, and the corresponding parameters are regarded as {\it a priori} data.
 \begin{itemize}
    \item The Lipschitz domain $\O$ is arcwise connected, with Lipschitz constant $L_\O>0$.
    \item The conductivity contrast $k>1$, i.e., the inclusion is more conductive than the surrounding medium.
    \item The inclusions are at distance at least $\d_0>0$ from the boundary $\p\O$, i.e., $D\subset \O_{\d_0}$ where $\O_{\d_0}:=\setdef{x\in\O}{\dist(x,\p\O)>\d_0}$.
    \item The interior domain $\O_{\d_0}$ is also arcwise connected.
    \item Each inclusion $D$ is a convex polyhedron with edges of length at least $l_0>0$.
    \item The openings of all vertices are assumed to lie in a fixed range $[\vartheta_m,\vartheta_M]\subset (0,2\pi)$.
    \item The openings of all edges are assumed to lie in a fixed range $[a_m,a_M]\subset (0,\pi)$.
    \item The Dirichlet boundary condition $f\in H^{1/2}(\p\O)$ is nontrivial, with $\|f\|_{H^{1/2}(\p\O)}>0$.
 \end{itemize}
 \end{assumption}
 
 \medskip

 \noindent{\bf Notations.} 
 \begin{itemize}
    \item We will use the symbols ``$\lesssim$'' and ``$\gtrsim$'' to denote the inequality up to a multiplicative constant depending only on {\it a priori} data.
    \item We denote by $d_H(D,D')$ the Hausdorff distance between $D$ and $D'$, which is defined as follows,
    \begin{align*}
        d_H(D,D')=\max\left\lbrace \sup_{x\in D}\dist(x,D'),\; \sup_{x'\in D'}\dist(x',D) \right\rbrace.
    \end{align*}
 \end{itemize}  

 \medskip

 The following theorem is the principal result of this paper.
 \begin{theorem}\label{th:main}
    Assume the geometrical and analytical conditions \zcref{assumption:geometric,assumption:singularity} hold. Let $u$ and $u'$ be solutions to \eqref{eq:calderon} with respect to $D$ and $D'$. We suppose that $\|\p_\nu(u-u')\|_{H^{-1/2}(\Gamma_0)}\leq \e$ for $\e>0$ being sufficiently small. Then, there exists positive constants $\kappa$ and $C$ depending only on {\it a priori} data such that
    \begin{align}\label{eq:stable}
        d_H(D,D')\leq C \left(\ln|\ln \e|\right)^{-\kappa}.
    \end{align}
 \end{theorem}
 We present \zcref{assumption:singularity} in more detail in the next section.

 \medskip

\subsection{Context positioning and technical ingredients}
The inverse problem considered in this paper is a special case of the classical Calder\'on problem. Problem \eqref{eq:calderon} was first proposed by Calder\'on in his 1980 seminar talk \cite{Calderon1980} and has been extensively studied over the past decades. Equation \eqref{eq:calderon} models the steady state of the electric potential $u$ in a conducting medium with conductivity distribution $\sigma\in L^\infty(\O)$. Calder\'on asked whether one can determine the unknown conductivity $\sigma$ from boundary measurements of the voltage $u$ and current flux $\p_\nu u$ on $\p\O$. This question initiated major developments in PDE inverse problems. Its importance comes not only from its mathematical challenges but also from broad applications in areas such as medical imaging, geophysics, and non-destructive testing. A representative application is electrical impedance tomography (EIT), which aims to reconstruct the conductivity distribution inside a body from boundary data. We refer to \cite{Ammari2008,Uhlmann2009} for comprehensive introductions to the theory and applications, and to \cite{Martins2019,Leitzke2020,Mansouri2021} for recent developments.

Broadly speaking, research on inverse problems focuses on uniqueness, stability, and numerical reconstruction. The uniqueness question asks whether the unknown parameter is uniquely determined by the measurements. The stability question asks how reconstruction errors depend on measurement errors. By the pioneering works of Sylvester and Uhlmann \cite{Sylvester1987} and Alessandrini \cite{Alessandrini1988}, smooth conductivities $\sigma$ can be uniquely and stably recovered from the Dirichlet-to-Neumann map, that is, from infinitely many boundary conditions in \eqref{eq:calderon}. However, from an application viewpoint, only finitely many measurements are typically available. Moreover, piecewise constant conductivities with jumps across interfaces are often more realistic for inhomogeneous media. These two facts motivate the framework studied in this paper: piecewise constant conductivity with finitely many measurements.

For inverse problems that determine the inclusion $D$ in \eqref{eq:transmission} from a single measurement, general uniqueness and stability results remain open. Existing results rely strongly on geometric features of $D$. In two dimensions, uniqueness for polygonal inclusions was established in \cite{Friedman1989,Seo1995}, and stability was obtained recently in \cite{moi5}. Also in two dimensions, stability for disks was studied in \cite{Fabes1999}; in \cite{moi2}, the authors improved the stability estimate and proposed a numerical reconstruction scheme. Another two-dimensional uniqueness result \cite{moi6} replaces polygonal inclusions with smooth boundaries having high-curvature points. In three dimensions, uniqueness for balls was addressed in \cite{Kang1999ball}. For convex polyhedral inclusions, uniqueness under well-chosen boundary data was obtained in \cite{Barcelo1994}, while the corresponding stability is the main objective of the present paper. It is also worth mentioning that, in a series of systematic studies \cite{Beretta2022,Beretta2021,Aspri2022}, Lipschitz-type stability was established using the Dirichlet-to-Neumann map for two-dimensional polygons and three-dimensional polyhedra. Furthermore, for general inclusion shapes, local stability results in \cite{Bellout1992,Choulli2003,Choulli2006} show that if two inclusions are sufficiently close (in the sense that one boundary is a small $\Cont^1$ perturbation of the other), then Lipschitz stability can be obtained. Combining Beretta's global Lipschitz stability with this local result allows one to reduce the number of measurements to two \cite{Hanke2024} for polygonal inclusions, at the cost of assuming that the number of vertices is known a priori. In this paper, we do not make that assumption and we do not use the Dirichlet-to-Neumann map.

We now introduce the three main ingredients of our proof: singularity decomposition, propagation of smallness, and microlocal analysis. Singularity decomposition extracts the explicit local behavior of solutions to elliptic equations near polyhedral interfaces. Propagation of smallness provides estimates of solution differences near the inclusion in terms of the measurement error. To derive the final stability estimate, we perform a microlocal analysis in a carefully selected region with a suitably chosen test function.

It is well known that the regularity of solutions to elliptic equations depends strongly on the smoothness of the boundary. In the presence of corners or edges, solutions exhibit singular behavior near these non-smooth regions. The study of such singularities was initiated by Kondratiev in the 1960s \cite{Kondratev1967}. Since then, a rich theory has been developed to describe the singular structure of solutions in non-smooth domains; see, for example, \cite{Grisvard1985,Dauge1988,Kozlov2002,Mazya2010}. In general, a solution can be decomposed into a singular part and a regular part. The regular part has regularity unaffected by boundary non-smoothness, while the singular part is expressed as a linear combination of explicitly constructed singular functions determined by the local geometry and boundary conditions. We mention here several relevant results for transmission problems \cite{Kellogg1971,Nicaise1990,Nicaise1994i,Nicaise1994ii,Dauge1989Oblique,Rempel1989,Dauge1999,Costabel1999}.

The other two ingredients are unique continuation and microlocal analysis. Unique continuation addresses Cauchy problems for elliptic equations with partial boundary data; we refer to \cite{Alessandrini2009} and references therein. Our microlocal argument is mainly inspired by ideas from inverse scattering in \cite{Blsten2014,Elschner2015}. The method consists of constructing a special test function with a tunable parameter, inserting it into the variational formulation, and carefully estimating each integral term. The final stability estimate is then obtained by an appropriate choice of this parameter. This technique has led to several single-measurement uniqueness and stability results in inverse scattering; see \cite{Blsten2017,Blasten2016,Blsten2020,Liu2022,Cakoni2021}.

\subsection{Structure of the paper} 

In the following sections, we prove the stability result \eqref{eq:stable}. In \zcref{sec:singular}, we present singular decompositions for solutions to the transmission problem \eqref{eq:transmission} near a polyhedral corner. In \zcref{sec:propa}, we analyze unique continuation and derive estimates for $u-u'$ and $\n u-\n u'$ in terms of small Cauchy-data discrepancies on $\Gamma_0$. In \zcref{sec:integral}, we perform a microlocal analysis near a selected edge, establish an integral identity, and estimate each term. Finally, in \zcref{sec:proof}, we complete the proof by combining the previous results.

\section{Singular decomposition near a polyhedral corner}\label{sec:singular}

The following theorem appears in Theorem 7.1 of \cite{Dauge1999}, Theorem 5 of \cite{Nicaise1999}, and \cite{Costabel1999}.

\begin{theorem}\label{th:SingDecom}
   Let $u$ be the solution to \eqref{eq:transmission}. Then, $u$ admits the following decomposition into vertex-singular part, edge-singular part and a regular part.
   \begin{align}\label{eq:SingDecom}
        u=\sum_{{\bf v}\;{\rm vertex}} u_{{\bf v}}+\sum_{E \;{\rm edge}} u_E+u_{\rm reg}.
    \end{align}
\end{theorem}
The structure of singular functions is intricate. In what follows, we describe in detail the regularity of the regular part $u_{\rm reg}$, as well as the local behavior of singular functions $u_{{\bf v}}$ and $u_E$ near vertices and edges.
    \begin{enumerate}
        \item Let $s>0$. This parameter originally describes the regularity of the right-hand side in \eqref{eq:transmission} with homogeneous boundary conditions. It characterizes the regularity that $u$ may have due to possible source terms. As we will see below, $s$ also determines the maximal order in the singular expansion. In our case, by a standard lifting argument, we can choose $s$ to achieve the desired regularity of $u_{\rm reg}$. We discuss the choice of $s$ later.
        \item The regular part $u_{\rm reg}\in H^{s+1}(D)\cup H^{s+1}(\O\setminus \overline{D})$ and satisfies the same transmission conditions as in \eqref{eq:transmission}. Moreover, the following estimate holds in general,
        \begin{align}\label{eq:esti-reg-cinqdemi}
            \|u_{\rm reg}\|_{H^{s+1}(D)}+\|u_{\rm reg}\|_{H^{s+1}(\O\setminus \overline{D})}\leq C \|f\|_{H^{1/2}(\p\O)}\lesssim 1.
        \end{align}\item For each vertex $\bf v$ of $D$, the singular function $u_{\bf v}$ can be expressed as follows.
         \begin{align*}
            u_{\bf v}(\rho,\vartheta)=\eta(\rho)\sum_{\begin{subarray}{c}
            \lm\in \tilde{S}({\bf v})\\ -\frac{1}{2}<\Re(\lm)<s-\frac{1}{2}
        \end{subarray}} \rho^\lm \sum_{m=0}^{K_{\bf v,\lm}}\g_{{\bf v},\lm}^m \log^m(\rho)\Phi_{{\bf v},\lm}^m(\vartheta).
         \end{align*}
         \begin{itemize}
            \item $(\rho,\vartheta)\in \R_+\times \S^2$ represent spherical coordinates centered at the vertex ${\bf v}$.
            \item $\eta$ is a smooth cut-off function such that $\eta(0)=1$ and its support is sufficiently small. 
            \item $\lm$ is called the singular exponent associated with the vertex ${\bf v}$. The set $\tilde{S}({\bf v})$ corresponds to the poles of an operator-valued meromorphic function, referred to as Mellin symbols or pencil operators.
            \item The $\log$ terms appear only when $\lm$ is an integer. We will later impose a generic condition that $K_{{\bf v},\lm}=0$ for the singular exponents $\lm$ under consideration.
            \item The singularity coefficients $\g_{{\bf v},\lm}^m\in\C$ depend continuously on the input data and the geometry of the polyhedron $D$. In fact, the coefficients $\g_{{\bf v},\lm}^m$ can be expressed as duality pairings between the input data and associated singular functions of the adjoint problem. Such formulas can be found, for example, in \cite{Nicaise1990}. Moreover, the following estimate holds \cite{Nicaise1994i,Nicaise1994ii},
            \begin{align*}
                |\g_{{\bf v},\lm}^m|\leq C \|f\|_{H^{1/2}(\p\O)}\lesssim 1.
            \end{align*}
            \item The singular functions $\Phi_{{\bf v},\lm}^m \in H^1(\S^2)$ are harmonic in each subdomain of $\S^2$ and satisfy the same transmission conditions across interfaces. We discuss in more detail below the characterization of the singular exponents $\lm$ and the associated functions $\Phi_{{\bf v},\lm}^m$.
        \end{itemize}
        \item For each edge $E$ of $D$, the singular function $u_E$ can be expressed as follows.
        \begin{align*}
            u_E(r,\t,z)=\eta(\tilde{r})\sum_{\begin{subarray}{c}
            \mu\in \tilde{S}(E)\\ 0<\Re(\mu)<s
        \end{subarray}}\tilde{r}^\mu\sum_{m=0}^{K_{E,\mu}}\mathcal{K}[\g_{E,\mu}^m](\tilde{r},\widetilde{z}) \log^m(\tilde{r}) \vf_{E,\mu}^m(\t).
        \end{align*}
        \begin{itemize}
            \item $(r,\t,z)$ are cylindrical coordinates with the $z$-axis coinciding with the edge $E$.
            \item $\ds \tilde{r}=\frac{r}{\d(z)}$, where $\d$ is a smooth function equivalent to the distance from the endpoints of the edge $E$.
            \item $\mu$ is called the edge singular exponent and is obtained by arguments similar to those for vertex singular exponents.
            \item The edge singularity coefficients $\g_{E,\mu}^m$ vary in the transversal direction along the edge $E$. They are described as elements in (Kondratiev-type) weighted Sobolev spaces.
            \item The behavior of $\g_{E,\mu}^m$ near vertices is influenced by vertex singularities, we will discuss later in ``8.) interaction between vertex and edge singularities''.
            \item The operator $\mathcal{K}$ is defined by convolution with a Schwartz kernel. It acts as a lifting operator from the edge $E$ to its cylindrical neighborhood. For any fixed $\tilde{r}\neq 0$, the map $z\mapsto \mathcal{K}[\g](\tilde{r},z)$ is smooth on $E$. For any fixed $z\in E$, $\mathcal{K}[\g](\tilde{r},z)$ converges to $\g$ as $\tilde{r}\rightarrow 0$.
            \item Remarks analogous to those for the vertex coefficients $\g_{{\bf v},\lm}^m$ apply to the computation and estimates of $\g_{E,\mu}^m$. In particular, one has \cite{Costabel1993i},
            \begin{align}\label{eq:esti-edge-coef}
                \|\g_{E,\mu}^m\|_{V_{-s}^{s-\Re(\mu)}(E)}\leq C \|f\|_{H^{1/2}(\p\O)}\lesssim 1.
            \end{align}
            \item The same remarks on $\log$ terms as in the vertex case apply here. We will also impose a generic condition $K_{E,\mu}=0$ for the edge singular exponents $\mu$ under consideration.
            \item $\vf_{E,\mu}^m\in H^1(\S^1)$ is harmonic in each subinterval and satisfies the same transmission condition across the interfaces.
            \item For a more detailed description of edge singular functions, we refer to \zcref{explain:edge}.
        \end{itemize}
        \item The cut-off function $\eta$ equals $1$ on $[0,\ell]$ and vanishes on $[2\ell,+\infty)$, where $\ell$ is sufficiently small. For instance, we can choose
        \begin{align}\label{eq:cutoff-support}
            \ell=\min_{\begin{subarray}{c} E,E'\; {\rm edge}\\ {\bf v}\; {\rm vertex} \end{subarray}}   \left\lbrace \frac{1}{2},\; \frac{1}{5}l_0\;, \tan \left(\frac{\Theta_{{\bf v},E,\;E'}}{2}\right) \right\rbrace.
        \end{align}
        Here, $l_0$ denotes the smallest edge length and $\Theta_{{\bf v},E,E'}$ represents the opening angle between two adjacent edges $E$ and $E'$ meeting at the vertex ${\bf v}$. With this choice, the supports of vertex singular functions are pairwise disjoint, and the same holds for edge singular functions. We remark that $\ell$ depends only on geometric parameters of the polyhedron $D$, and can therefore be treated as an {\it a priori} parameter.
        \item To determine the vertex singular exponents $\lm$, we use the fact that the singular functions $\rho^\lm \Phi_{{\bf v},\lm}(\vartheta)$ satisfy the same transmission problem in a sufficiently small ball centered at ${\bf v}$. Writing the transmission problem in spherical coordinates centered at ${\bf v}$, we deduce that $\Phi_{{\bf v},\lm}$ satisfies
        \begin{align*}
            \dvg_{\S^2}((1+(k-1)\chi_{G_{\bf v}})\n_{\S^2}\Phi_{{\bf v},\lm})+\lm(\lm+1)\Phi_{{\bf v},\lm}=0 \quad {\rm on}\;\; \S^2.
        \end{align*}
        Here, $C_{\bf v}$ denotes the polyhedral cone generated by the vertex ${\bf v}$ and its adjacent edges, and $G_{\bf v}=\S^2\cap C_{\bf v}$ is the corresponding spherical polygon. Then the eigenvalues $\lm$ can be expressed as follows.
        \begin{align*}
            \lm=-\frac{1}{2}\pm\sqrt{\frac{1}{4}+\upsilon},
        \end{align*}
        where $\upsilon$ describes the eigenvalues of the associated Laplace-Beltrami operator acting on $H^1(\S^2)$:
        \begin{align*}
            \Lop_{\bf v}: u\mapsto -\dvg_{\S^2}((1+(k-1)\chi_{G_{\bf v}})\n_{\S^2} u),
        \end{align*}
        where the derivatives are understood in weak sense. This elliptic operator is unbounded, positive, self-adjoint with dense domain, and $(Id-\Lop_{\bf v})^{-1}$ is compact. From spectral theory, its eigenvalues can be ordered as $\upsilon_1\leq \upsilon_2\leq \cdots$ with $\upsilon_j\rightarrow +\infty$ when $j\rightarrow +\infty$. The first eigenvalue $\upsilon_1$ can be characterized by the following Rayleigh quotient:
        \begin{align*}
            \upsilon_1=\min_{u\in H^1(\S^2),\; u \perp {\rm const}}\frac{\ds \int_{\S^2}|\n_{\S^2} u|^2+(k-1)\int_{G_{\bf v}}|\n_{\S^2} u|^2}{\ds \int_{\S^2}|u|^2}.
        \end{align*}
        We can obtain, by using the Poincar\'e inequality, that $\upsilon_1 > \min\{1,k\}\Lambda_1$, where $\Lambda_1$ is the first positive eigenvalue of the Laplace-Beltrami operator $\Delta_{\S^2}$. Direct calculation shows that $\Lambda_1=2$. Consequently, the vertex singular exponents follow the order,
        \begin{align}\label{eq:les-lm}
            1<\lm_1 < \lm_2 < \cdots \rightarrow +\infty.
        \end{align}
        \item The edge singular exponents $\mu$ are exactly the same as the corner singularity of a plane sector in $\R^2$. For an edge of an opening angle $a\in (0,\pi)$, the singularity exponents $\mu$ satisfies, (see e.g. (2.31) in \cite{Bellout1992}) 
        \begin{align}\label{eq:edge-expo}
            \sin(\mu(\pi-a))=\left|\frac{k+1}{k-1}\right|\sin(\mu \pi).
        \end{align}
        The associated functions $\vf_{E,\mu}^m$ are simply $\cos(\mu\t+\phi_\pm)$ with two phases $\phi_\pm$ such that the transmission conditions hold. Using the same argument as for vertex singularities, these exponents can be ordered increasingly. Moreover, by studying the behavior of the function $t\mapsto \frac{\sin(t(\pi-a))}{\sin(t\pi)}$, one deduces that,
        \begin{align}\label{eq:les-mu}
            \frac{1}{2}<\mu_1<1<\mu_2<\cdots \rightarrow +\infty.
        \end{align}
        \item Interaction between vertex and edge singularities. Another crucial point is how edge-singularity coefficients behave near an adjacent vertex. For simplicity, we consider the case without $\log$ terms as explained in \zcref{assumption:singularity}. Let $\mu$ be an edge singular exponent and $\g_{E,\mu}$ be the corresponding edge singular coefficient. Let $\rho$ denote the distance from the vertex $\bf v$ to a point on the edge $E$. Then, the edge singular coefficient $\g_{E,\mu}$ has the following asymptotic behavior \cite{Dauge1999} near the vertex $\bf v$,
        \begin{align}\label{eq:asym-coef}
            \g_{E,\mu}(\rho)= \sum_{\lm } a_{E,\mu}^{{\bf v},\lm}\rho^{\lm}+o (\rho^{\lm}) \quad {\rm as}\;\; \rho\rightarrow 0,
        \end{align}
        where $a_{E,\mu}^{{\bf v},\lm}$ are constant coefficients and $\lm=\lm_1,\lm_2,\cdots$ are vertex singular exponents.
    \end{enumerate}

    \medskip

    We now impose some generic assumptions to simplify the stability analysis. The exceptional cases are very limited and can be avoided by a suitable choice of boundary condition. We introduce the smallest singularity exponent,
        \begin{align}\label{eq:def-s0}
            s_0=\min_{\begin{subarray}{c} {\bf v}\; {\rm vertex} \\ E\; {\rm edge} \end{subarray}} \left\lbrace\lm_{{\bf v}}+\frac{1}{2},\;\;\mu_{E} \right\rbrace.
        \end{align}

    \begin{assumption}[Singularity]\label{assumption:singularity}
        Throughout the paper, we will assume that the following conditions hold.
    \begin{enumerate}
        \item  We follow case b) of Theorem 7.1 in \cite{Dauge1999}. That is, the parameter $s$ is chosen such that $s_0\leq s <s_0+1$ in order to have sufficient regularity for $u_{\rm reg}$ and to avoid the ``shadow terms'' in condition c) for singular functions. Furthermore, from the estimates of the first singularity exponents $\frac{1}{2}<\mu_1<1<\lm_1$, we have $s_0=\mu_1>\frac{1}{2}$ and $s_0+1>\frac{3}{2}$. With this condition, we will choose $s$ close enough to $\mu_1+1$ such that,
        \begin{align*}
            \frac{3}{2}< s <\mu_1+1.
        \end{align*}
        As an immediate consequence, the regular part $u_{\rm reg}$ has at least $H^{5/2+\epsilon}$ regularity in each subdomain, where $\epsilon:=s-\frac{3}{2}>0$.
        \item There is no $\log$ terms in the vertex or edge singular expansions. That is, $K_{\bf v,\lm}=K_{E,\mu}=0$ for any vertex $\bf v$ and any edge $E$.
        \item For any edge $E$, the first edge singularity is not degenerate, i.e., $\g_{E,\mu_1}$ is not identically null on $E$.
        \item For any adjacent vertex-edge pair $({\bf v}, E)$ the first interaction coefficient $a_{E,\mu_1}^{{\bf v},\lm_1}$ in the asymptotic expansion \eqref{eq:asym-coef} is not zero.
    \end{enumerate}
    \end{assumption}

    \subsection*{Discussion on the singularity assumptions}
    \begin{enumerate}
        \item The condition imposed on $s$ is more a matter of convenience than a strict necessity. The main purpose is to ensure that the regular part $u_{\rm reg}$ has sufficient regularity for the subsequent analysis. In fact, one can let $s$ tend to infinity to obtain a full expansion as in Lemma 2.1 of \cite{Bellout1992}. The present choice of $s$ suffices for our analysis, and avoiding the unnecessary complexity of shadow terms is a practical advantage.
        \item This condition corresponds exactly to Remark 5.3 in \cite{Dauge1999}. From that paper, the $\log$ terms appear only when the singular exponent is an integer. In our case, we consider only the first singular exponents $\lm_1$ and $\mu_1$, which are not integers by the previous estimates \eqref{eq:les-lm} and \eqref{eq:les-mu}. Therefore, this condition is generically satisfied.
        \item From the estimates \eqref{eq:les-lm} and \eqref{eq:les-mu}, only the first edge singular exponent $\mu_1$ is less than $1$. From the singular decomposition, the blow-up phenomenon of $\n u$ occurs only on non-degenerate edges. If the first edge singularity is degenerate, then the gradient $\n u$ is bounded in a neighborhood of the edge. Since the blow-up of electric fields near sharp boundaries of embedded inclusions is a well-recognized physical phenomenon, we treat this condition as a generic assumption.  
        \item This condition states implicitly that the first vertex singularity is also non-degenerate for every vertex ${\bf v}$. One can see in \cite{Dauge1999} how vertex and edge singularities interact. In fact, this condition can be relaxed to a weaker version, which implies a weaker stability estimate.
        \item The question of which boundary data guarantee the above assumptions is an interesting direction for further research, though it falls outside the main scope of this paper. We remark that in \cite{Kang1997,Alessandrini1992,Alessandrini1994}, such a boundary condition is explicitly constructed, from which stability is derived for an inverse conductivity problem under a single measurement. However, their construction is based on the two-dimensional case and cannot be directly extended to three dimensions. We will leave this question for future research.
    \end{enumerate}

    \section{Propagation of smallness}\label{sec:propa}

    In this section, we establish estimates for $u-u'$ as well as $\n u-\n u'$ in terms of the small discrepancy in the Cauchy data on $\p\O$. The main results here are essentially the same as those obtained in \cite{moi5}. To keep the paper reasonably self-contained, we list the key ingredients used to derive these estimates. We begin with the classical three-sphere inequality for harmonic functions \cite{Korevaar1994}.

    \begin{lemma}[Three sphere inequality]
        Let $0<r_1<r_2<r_3$. There exists a constant $\widetilde{\a}\in (0,1)$ such that for any harmonic function $w\in H^1_{loc}(\R^n)$ with $n\geq 2$, it holds that
        \begin{align*}
            \sup_{|x|\leq r_2}|w(x)|\leq \left(\sup_{|x|\leq r_1}|w(x)|\right)^{\widetilde{\a}}\left(\sup_{|x|\leq r_3}|w(x)|\right)^{1-\widetilde{\a}}.
        \end{align*}
        The exponent $\widetilde{\a}$ depends only on $n$ and the ratios $r_1/r_2$, $r_2/r_3$.
    \end{lemma}

    From now on, we fix the ratios $r_2/r_1=r_3/r_2=2$, and hence the exponent $\widetilde{\a}$ can be treated as an {\it a priori} parameter.

    \medskip

    Next, we apply the three-sphere inequality iteratively along a chain of balls with suitably chosen radii, in order to estimate $w$ near an interior point in terms of its size near the boundary. In what follows, $B(x,r)$ denotes the open ball centered at $x$ with radius $r$.

    \begin{proposition}\label{prop:p1}
        Let $r>0$ and let $U\subset \R^3$ be a bounded and connected domain. Consider $x,y\in U$ and assume that there exists a rectifiable curve $\g$ in $U$ with endpoints $x$ and $y$ such that $\dist(\g,\p U)>4r$. Consider also a bounded harmonic function $w$ in $U$. We assume furthermore that $w$ is bounded with $\|w\|_{L^\infty(B(x,r))}\leq 1$ and $\|w\|_{L^\infty(U)}= T>0$. Then, it holds that
        \begin{align*}
            \|w\|_{L^\infty(B(y,r))}\leq T \|w\|_{L^\infty(B(x,r))}^{\widetilde{\a}^{|\g|/r+1}},
        \end{align*}
        where $|\g|$ denotes the length of $\g$.
    \end{proposition}
    \begin{proof}
        See Lemma 3.2 in \cite{moi5}.
    \end{proof}

    \medskip

    The next step is to formulate propagation of smallness from boundary data into the interior. We obtain the following result.

    \begin{proposition}\label{prop:aa2}
        Let $U$ be a Lipschitz connected domain. Let $\Sigma\subset \p U$ be an open portion of the boundary. Consider a harmonic function $w$ in $U$. We make the following assumptions 
        \begin{itemize}
            \item for any $y\in U$, there exists a rectifiable curve $\g$ in $U$ connecting $y$ and a point $x\in \Sigma$,
            \item the length of $\g$ is bounded by a uniform constant $L>0$,
            \item $\g$ does not touch the boundary of $U$ outside $\Sigma$, i.e., $\dist(\g,\p U\setminus \Sigma)>4r$ for some $r>0$,
            \item $w$ is bounded with $\|w\|_{L^\infty(U)}=T>0$,
            \item $\|w\|_{H^{1/2}(\Sigma)}+\|\p_{\nu} w\|_{H^{-1/2}(\Sigma)}\leq \e$ with $0< \e<1$.
        \end{itemize}
        Then, for all $y\in U$ and all sufficiently small $r>0$, one has
        \begin{align}\label{eq:propa-small-holder}
            |w(y)|\leq C T \e^{\widetilde{\a}^{L/r+1}},
        \end{align}
        with a constant $C$ depending only on {\it a priori} data.
    \end{proposition}
    \begin{proof}
        This follows from the arguments in Remark 6.3 of \cite{Alessandrini2009}. Using Lemma 6.1 and Theorem 6.2 in that paper, we can extend $w$ to an exterior neighborhood $V$ of $\Sigma$ such that the extension $\widetilde{w}$ satisfies $\|\widetilde{w}\|_{H^1(V)}\leq C\|w\|_{H^{1/2}(\Sigma)}$. Moreover, $\widetilde{w}$ solves $\Delta \widetilde{w}=\widetilde{f}$ in $U\cup V$ with a sufficiently small right-hand side $\widetilde{f}$. Then, interior elliptic regularity yields $\|\widetilde{w}\|_{L^\infty(B(x_0,r_0))}\leq \e$ for some $x_0\in V$ and $r_0>0$, where $r_0$ depends only on the Lipschitz character of $\Sigma$. Finally, we apply \zcref{prop:p1} to this modified equation to deduce the desired estimate.
    \end{proof}

    \medskip

    We now state the estimate needed for $u-u'$ and $\n(u-u')$ in the region where the analysis is carried out. Since $D$ and $D'$ are convex polyhedra, the Hausdorff distance $d_H(D,D')$ is attained at some vertex $\bf v$, which may be assumed to be a vertex of $D$ without loss of generality.

    \begin{proposition}\label{prop:propa-small}
        Let $u$ and $u'$ be solutions to \eqref{eq:calderon} corresponding to $D$ and $D'$, respectively. Let ${\bf v} \in \p D$ be a vertex of $D$, and let $h>0$ satisfy $B({\bf v},h)\subset\O_{\d_0}$ and $B({\bf v},h)\cap D'=\varnothing$. Then there exist $\a \in (0,1)$ and $T>0$, depending only on {\it a priori} data, such that for $0<\e<\e_m$,
        \begin{align*}
            \forall x\in B({\bf v},h)\setminus \overline{D},\quad & |(u-u')(x)|\leq T\left(\ln |\ln \e|\right)^{-\a};\\
            &|\n (u-u')(x)|\leq \frac{T}{\dist(x,\p D)}\left(\ln |\ln \e|\right)^{-\a}.
        \end{align*}
    \end{proposition}
    \begin{proof}
        Using the convexity of $D$ and $D'$, there exist affine hyperplanes $P'$ and $P$ such that $P'$ separates $B({\bf v},h)$ from $D'$, and $P\cap \overline{(D\cup D')}={\bf v}$. Denote by $P'_\pm$ the two half-spaces determined by $P'$, with $D'\subset P'_-$. Similarly, let $P_-$ be the half-space determined by $P$ such that $D\cup D'\subset P_-$, and let $P_+$ be the opposite half-space. Then the region $V=(P_+\cap P'_+ \cap \O_{\d_0})\cup (\O\setminus \overline{\O_{\d_0}})$ is connected and satisfies $\Gamma_0\subset \p V$. Consequently, we can find a rectifiable curve $\g_{0,{\bf v}}$ connecting $\bf v$ to a point $x_0\in \Gamma_0$ such that $\g_{0,{\bf v}}$ stays at distance at least $r_m:=\frac{1}{3}\min\{h,\d_0\}$ from $\p\O\cup \p (D\cup D')\setminus (\Gamma_0 \cup B({\bf v},h))$. We also define the domain $U$ by
        \begin{align*}
            U=B({\bf v},h)\cup \setdef{y\in\O}{\dist(y,\g_{0, {\bf v}})<r_m}.
        \end{align*}        
        Then we choose $r<\min\{r_0,r_m\}$ small enough to apply \zcref{prop:aa2} to the harmonic function $u-u'$ in the domain $U$, along a curve obtained by slightly modifying $\g_{0,{\bf v}}$. Here the quantity $r_0$ depends only on the Lipschitz character of $\Sigma$, which comes from the proof of \zcref{prop:aa2}. It remains to derive H\"older estimates for $u-u'$ and its gradient on $U$ in terms of the distance to $\p D$.

        The well-known result of De Giorgi-Nash-Moser (Theorem 8.24 in \cite{Gilbarg-Trudinger}) states that solutions to strongly elliptic equations with bounded coefficients are H\"older continuous at interior points. Consequently, $u-u'\in\Cont^{\a_0}(\O_{\d_0})$ for some $\a_0>0$. We also introduce $T_0=\|u-u'\|_{\Cont^{\a_0}(\O_{\d_0})}$.
        
        For the gradient $\n (u-u')$, we expect blow-up near the boundary $\Gamma_{D\cup D'}$. To control the blow-up rate, we use the decomposition of $u$ and $u'$ into singular and regular parts as stated in \zcref{th:SingDecom}.
        
        We recall that the regular part $u_{\rm reg}$ has $H^{5/2+\epsilon}$ regularity in $D\cup (\O\setminus \overline{D})$. By applying Sobolev embedding $H^{3/2+\epsilon}\hookrightarrow \Cont^{\epsilon}$ in $\R^3$, we can deduce that $\n u_{\rm reg}$ is H\"older continuous in $U$. 
        
        For the singular parts, it is sufficient to work in local spherical or cylindrical coordinates near vertices or edges. Using the estimates of the singularity exponents $\lm$ and $\mu$, we deduce that the singularity of $\n u_{\bf v}$ is at most of type $\rho^{\lm_1-1}$ and the singularity of $\n u_E$ is at most of type $r^{\mu_1-1}$, as shown in \zcref{lemma:esti-grad-edge}. Consequently, the function $x\mapsto \dist(x,\p D)\n u(x)$ is $\Cont^{\a_1}$ H\"older continuous in $U$ with $0<\a_1\leq \min\{\epsilon,\lm_1, \mu_1\}$.
        
        The same arguments hold for $u'$. Hence, we can deduce that $\dist(x,\p D)\n (u-u')(x)$ is H\"older continuous in $U$, and thus we can set $T_1=\|\dist(\cdot,\p D)\n (u-u')\|_{\Cont^{\a_1}(U)}$. Next, we set $\a=\min\{\a_0,\a_1\}$ and $T=\max\{T_0,T_1\}$ and remark that those parameters depend only on {\it a priori} data. 
        
        The remainder of the proof consists in estimating \eqref{eq:propa-small-holder} as $r\rightarrow0^+$ by using the H\"older continuity of these functions. The argument follows Proposition 3.4, Corollary 3.5, and Proposition 3.7 in \cite{moi5}. The value of $\e_m$ is given by (3.4) in \cite{moi5}.
    \end{proof}

    \section{Crucial local integral estimates}\label{sec:integral}

    \subsection{Geometrical Settings}

        From now on, $D$ and $D'$ denote two admissible polyhedra. By an elementary geometric argument, the Hausdorff distance $d_H(D,D')$ is attained at some vertex ${\bf v}$ of $D$ or $D'$. We carry out the analysis in a local cylindrical region around one of the edges adjacent to ${\bf v}$; see \zcref{fig:local-polyhedron-alt}. The following points describe the local geometric setting used in the analysis; see also \zcref{fig:local-cylindrical}.
        \begin{itemize}
            \item ${\bf v}$ is a vertex of $D$ which achieves the Hausdorff distance.
            \item $h>0$ is chosen small enough such that the region $\widetilde{\O}$ constructed below is at distance at least $h$ from the other polyhedron $D'$.
            \item $E$ is an edge of $D$ linked to the vertex ${\bf v}$.
            \item The opening of the edge $E$ is $a\in [a_m,a_M] \subset (0,\pi)$.
            \item $O$ is a point on $E$ with a distance $\rho>0$ to the vertex ${\bf v}$.
            \item $\Pi$ is the plane perpendicular to $E$ and passing through $O$.
            \item $(r,\t,z)$ denotes cylindrical coordinates with $z$-axis coinciding with $E$ and origin at $O$.
        \end{itemize}

    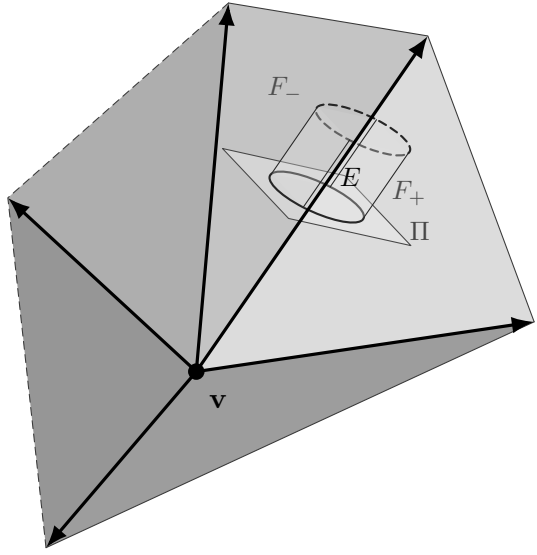
\begin{figure}[ht]
        \centering
        \begin{tikzpicture}[
            scale=2,
            >=Latex,
            line join=round,
            line cap=round,
            x={(1cm,0cm)},
            y={(0.28cm,1.00cm)},
            z={(-1.18cm,0.34cm)}
        ]
            \coordinate (V) at (0,0,0);
            \coordinate (A) at (2.55,0.22,0.32);
            \coordinate (B) at (1.65,2.02,0.58);
            \coordinate (C) at (-0.15,2.35,0.25);
            \coordinate (D) at (-1.95,1.25,-0.30);
            \coordinate (E) at (-1.28,-1.00,-0.48);
            \coordinate (O) at ($(V)!0.52!(B)$);
            \coordinate (Oc) at ($(V)!0.72!(B)$);

            \coordinate (P1) at ($(O)+(0.376,-0.231,-0.264)$);
            \coordinate (P2) at ($(O)+(-0.190,0.231,-0.264)$);
            \coordinate (P3) at ($(O)+(-0.376,0.231,0.264)$);
            \coordinate (P4) at ($(O)+(0.190,-0.231,0.264)$);

            \fill[black!62,opacity=0.66] (V) -- (D) -- (E) -- cycle;
            \fill[black!48,opacity=0.66] (V) -- (C) -- (D) -- cycle;
            \fill[black!56,opacity=0.68] (V) -- (E) -- (A) -- cycle;
            \fill[black!32,opacity=0.72] (V) -- (B) -- (C) -- cycle;
            \fill[black!18,opacity=0.76] (V) -- (A) -- (B) -- cycle;

            \fill[gray!15,opacity=0.45] (P1) -- (P2) -- (P3) -- (P4) -- cycle;
            \draw[gray!55!black] (P1) -- (P2) -- (P3) -- (P4) -- cycle;
            \begin{scope}[shift={(O)},x={($(P1)-(O)$)},y={($(P2)-(O)$)}]
                \draw[gray!30!black,thick] (0,0) circle[radius=0.48];
                \coordinate (C1) at (0.48,0);
                \coordinate (C2) at (0,0.48);
                \coordinate (C3) at (-0.48,0);
                \coordinate (C4) at (0,-0.48);
            \end{scope}
            \begin{scope}[shift={(Oc)},x={($(P1)-(O)$)},y={($(P2)-(O)$)}]
                \draw[densely dashed,gray!30!black,thick] (0,0) circle[radius=0.48];
                \coordinate (D1) at (0.48,0);
                \coordinate (D2) at (0,0.48);
                \coordinate (D3) at (-0.48,0);
                \coordinate (D4) at (0,-0.48);
            \end{scope}
            \fill[gray!22,opacity=0.18] (C1) -- (C2) -- (D2) -- (D1) -- cycle;
            \fill[gray!22,opacity=0.18] (C2) -- (C3) -- (D3) -- (D2) -- cycle;
            \fill[gray!22,opacity=0.18] (C3) -- (C4) -- (D4) -- (D3) -- cycle;
            \fill[gray!22,opacity=0.18] (C4) -- (C1) -- (D1) -- (D4) -- cycle;
            \draw[gray!30!black] (C1) -- (D1);
            \draw[gray!30!black] (C2) -- (D2);
            \draw[gray!30!black] (C3) -- (D3);
            \draw[gray!30!black] (C4) -- (D4);
            \node[gray!55!black] at ($(P1)+(0.03,0.10,0)$) {\small $\Pi$};

            \draw[gray!45!black] (A) -- (B) -- (C);
            \draw[densely dashed,gray!45!black] (C) -- (D) -- (E);
            \draw[gray!45!black] (E) -- (A);

            \draw[very thick,->] (V) -- (A);
            \draw[very thick,->] (V) -- (B);
            \draw[very thick,->] (V) -- (C);
            \draw[very thick,->] (V) -- (D);
            \draw[very thick,->] (V) -- (E);

            \fill (V) circle (1.55pt);
            \node at ($(V)!0.5!(A)!0.5!(E)+(0.00,0.3,0)$) {$\mathbf{v}$};

            \node at ($(V)!0.57!(B)+(0.14,0.02,0)$) {\small $E$};

            \node[black!70] at ($(V)!0.55!(A)!0.45!(B)+(0.02,0.08,0)$) {\small $F_+$};
            \node[black!70] at ($(V)!0.55!(B)!0.45!(C)+(0.00,0.10,0)$) {\small $F_-$};
            
        \end{tikzpicture}
        \caption{Polyhedral corner at vertex~$\mathbf{v}$.}
        \label{fig:local-polyhedron-alt}
    \end{figure}

        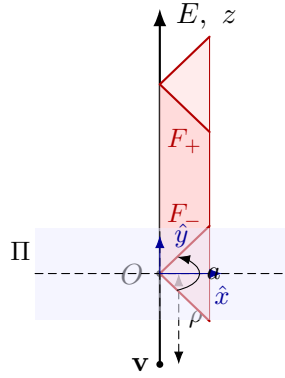
\begin{figure}[ht]
            \centering
            \begin{tikzpicture}[scale=1.0,>=Latex,line join=round,line cap=round]
                \def\ang{64}          
                \def\edgelen{4.2}     
                \def\rhovv{1.2}       
                \def\rr{1.6}          
                \def\thP{35}          

                \pgfmathsetmacro{\fx}{0.75}   
                \pgfmathsetmacro{\fy}{0.35}   

                \coordinate (V) at (0,0);
                \coordinate (O) at (0,\rhovv);
                
                \draw[thick,->] (V) -- (0,\edgelen+0.5);
                \node[right] at (0.08,\edgelen+0.4) {$E,\ z$};

                \fill (V) circle (1.3pt);
                \node[left] at (V) {$\mathbf{v}$};

                \fill (O) circle (1.3pt);
                \node[left] at (-0.08,\rhovv) {$O$};

                \draw[<->,densely dashed] (0.25,0) -- (0.25,\rhovv);
                \node[right] at (0.25,0.5*\rhovv) {\small $\rho$};

                \pgfmathsetmacro{\Gpx}{\fx*cos(\ang)}
                \pgfmathsetmacro{\Gpy}{\fy*sin(\ang)}
                \pgfmathsetmacro{\Gmx}{\fx*cos(-\ang)}
                \pgfmathsetmacro{\Gmy}{\fy*sin(-\ang)}

                \def\flen{2.0}  
                \def\zht{2.5}   

                \pgfmathsetmacro{\fGpx}{\flen*\Gpx}
                \pgfmathsetmacro{\fGpy}{\flen*\Gpy}
                \pgfmathsetmacro{\fGmx}{\flen*\Gmx}
                \pgfmathsetmacro{\fGmy}{\flen*\Gmy}

                \fill[red!12,opacity=0.6]
                    (O) -- ($(O)+(\fGpx,\fGpy)$)
                    -- ($(O)+(0,\zht)+(\fGpx,\fGpy)$)
                    -- ($(O)+(0,\zht)$) -- cycle;
                \fill[red!18,opacity=0.6]
                    (O) -- ($(O)+(\fGmx,\fGmy)$)
                    -- ($(O)+(0,\zht)+(\fGmx,\fGmy)$)
                    -- ($(O)+(0,\zht)$) -- cycle;

                \draw[thick,red!70!black] (O) -- ($(O)+(\fGpx,\fGpy)$);
                \draw[thick,red!70!black] (O) -- ($(O)+(\fGmx,\fGmy)$);
                \draw[thick,red!70!black] ($(O)+(0,\zht)$)
                    -- ($(O)+(0,\zht)+(\fGpx,\fGpy)$);
                \draw[thick,red!70!black] ($(O)+(0,\zht)$)
                    -- ($(O)+(0,\zht)+(\fGmx,\fGmy)$);
                \draw[red!70!black] ($(O)+(\fGpx,\fGpy)$)
                    -- ($(O)+(0,\zht)+(\fGpx,\fGpy)$);
                \draw[red!70!black] ($(O)+(\fGmx,\fGmy)$)
                    -- ($(O)+(0,\zht)+(\fGmx,\fGmy)$);

                \pgfmathsetmacro{\GpLabx}{0.5*\flen*\Gpx}
                \pgfmathsetmacro{\GpLaby}{0.5*\flen*\Gpy}
                \pgfmathsetmacro{\GmLabx}{0.5*\flen*\Gmx}
                \pgfmathsetmacro{\GmLaby}{0.5*\flen*\Gmy}
                \pgfmathsetmacro{\halfzht}{0.5*\zht}
                \node[red!70!black] at
                    ($(O)+(\GpLabx,\GpLaby)+(0,\halfzht+0.18)$)
                    {\small $F_+$};
                \node[red!70!black] at
                    ($(O)+(\GmLabx,\GmLaby)+(0,\halfzht-0.18)$)
                    {\small $F_-$};

                \pgfmathsetmacro{\pihalfx}{2.2*\fx}
                \pgfmathsetmacro{\pihalfy}{1.3*\fy}
                \pgfmathsetmacro{\pilabx}{2.1*\fx}
                \fill[blue!8,opacity=0.5]
                    ($(O)+(-\pihalfx,-\pihalfy-0.15)$) rectangle
                    ($(O)+(\pihalfx,\pihalfy+0.15)$);
                \draw[densely dashed]
                    ($(O)+(-\pihalfx,0)$) -- ($(O)+(\pihalfx,0)$);
                \draw[densely dashed]
                    ($(O)+(0,-\pihalfy-0.15)$) -- ($(O)+(0,\pihalfy+0.15)$);
                \node[above left] at ($(O)+(-\pilabx,0.05)$) {\small $\Pi$};

                \pgfmathsetmacro{\xhatlen}{1.1*\fx}
                \pgfmathsetmacro{\yhatlen}{1.1*\fy+0.12}
                \draw[->,blue!60!black] (O) -- ($(O)+(\xhatlen,0)$);
                \node[blue!60!black,below] at ($(O)+(\xhatlen,-0.02)$)
                    {\small $\hat{x}$};
                \draw[->,blue!60!black] (O) -- ($(O)+(0,\yhatlen)$);
                \node[blue!60!black,right] at ($(O)+(0.06,\yhatlen)$)
                    {\small $\hat{y}$};

                \pgfmathsetmacro{\aarcRx}{0.7*\fx}
                \pgfmathsetmacro{\aarcRy}{0.7*\fy}
                \pgfmathsetmacro{\aarcSx}{0.7*\Gmx}
                \pgfmathsetmacro{\aarcSy}{0.7*\Gmy}
                \draw[->] ($(O)+(\aarcSx,\aarcSy)$)
                    arc[start angle={-\ang},end angle={\ang},
                        x radius=\aarcRx,y radius=\aarcRy];
                \pgfmathsetmacro{\alabx}{0.95*\fx}
                \node at ($(O)+(\alabx,0.0)$) {\small $a$};
            \end{tikzpicture}
            \caption{Local cylindrical coordinates $(r,\theta,z)$ around the edge~$E$.
            The origin~$O$ lies on~$E$ at distance~$\rho$ from the vertex~$\mathbf{v}$.
            The plane~$\Pi$ is perpendicular to~$E$ at~$O$;
            the two faces $F_\pm$ meet along~$E$ with opening angle~$a$.}
            \label{fig:local-cylindrical}
        \end{figure}

        Next, we describe the region and surfaces involved in the integral estimates. We first define contours in the plane $\Pi$; the corresponding surfaces are obtained by translating these contours along the edge $E$. See \zcref{fig:cross-section}.

        \begin{itemize}
            \item $\Gamma_{\pm}:=\p D \cap \Pi$ denotes the two sides of the corner at $O$ in the plane $\Pi$. In local cylindrical coordinates, we impose that
            \begin{align*}
                \Gamma_{\pm}=\setdef{(r,\t,0)}{0<r<h,\; \t=\pm a/2}.
            \end{align*}
            \item From the choice made in the previous point, we also introduce the unit vectors $\hat{x}$ and $\hat{y}$ in $\Pi$ defined by $\t=0$ and $\t=\pi/2$, respectively.
            \item Let $\tau>0$ be a parameter, whose value will be chosen carefully to derive the desired stability estimate. Define $\p S^e$ as the contour in $\Pi \setminus D$ such that each point on $\p S^e$ has distance $\ds \frac{1}{\tau}$ to the sides $\Gamma_{\pm}$.
            \item The contour $\p S^e$ intersects the circle $\setdef{(r,\t,0)}{r=h}$ at two points. We denote by $\p S^i$ the arc of this circle between these intersection points. In local cylindrical coordinates, we have
            \begin{align*}
                 \p S^i=\setdef{(r,\t,0)}{r=h,\; -\frac{a}{2}-\arcsin(\frac{1}{\tau h})\leq \t\leq \frac{a}{2}+\arcsin(\frac{1}{\tau h})}.
            \end{align*} 
            \item $A_-$ denotes the region enclosed by $\p S^e$ and $\p S^i$. Also, $A_-^i$ and $A_-^e$ denote the interior and exterior parts of $A_-$ with respect to $D$, respectively.
            \item By translation along the edge $E$ with a distance $h_z>0$, we introduce the following surfaces and regions:
            \begin{align*}
                F_\pm &:=\setdef{(r,\t,z)}{(r,\t,0)\in \Gamma_\pm, \; 0<z<h_z},\\
                S^e &:=\setdef{(r,\t,z)}{(r,\t,0)\in \p S^e, \; 0<z<h_z},\\
                S^i &:=\setdef{(r,\t,z)}{(r,\t,0)\in \p S^i, \; 0<z<h_z},\\
                A_+^{i,e} &:=\setdef{(r,\t,h_z)}{(r,\t,0)\in A_-^{i,e}},\\
                \widetilde{\O}&:=\setdef{(r,\t,z)}{(r,\t,0)\in A_-, \; 0<z<h_z},\\
                \widetilde{D}&:=D\cap \widetilde{\O}=\setdef{(r,\t,z)}{(r,\t,0)\in A_-^i, \; 0<z<h_z}.
            \end{align*}
            \item The parameters $h,\rho$, and $h_z$ are chosen such that 
            \begin{align}                
                \rho+h+h_z &\leq \ell, \label{eq:constraint1}\\
                \frac{h}{\rho}&\leq \ell, \label{eq:constraint2}
            \end{align}
             where $\ell$ is given in \eqref{eq:cutoff-support}. With this choice, only singular functions associated with the vertex $\bf v$ and the edge $E$ appear in the local integral estimates.
        \end{itemize}

        \begin{figure}[ht]
            \centering
            \begin{tikzpicture}[scale=1.1,>=Latex,line cap=round,line join=round]
                \def\ang{64}        
                \def\R{2.4}         
                \def\dsep{0.38}     
           
                \pgfmathsetmacro{\thext}{\ang + asin(\dsep/\R)}

                \fill[gray!20,opacity=0.7]
                    ({\R*cos(\thext)},{\R*sin(\thext)})
                    -- ({-\dsep*sin(\ang)},{\dsep*cos(\ang)})
                    arc[start angle={90+\ang},end angle={270-\ang},radius=\dsep]
                    -- ({\R*cos(-\thext)},{\R*sin(-\thext)})
                    arc[start angle={-\thext},end angle={\thext},radius=\R]
                    -- cycle;
                \pgfmathsetmacro{\Amid}{0.5*(180+\thext)}
                \node at ({0.7*\R*cos(\Amid)},{0.7*\R*sin(\Amid)}) {$A_-$};

                \fill[red!10] (0,0)
                    -- ({\R*cos(\ang)},{\R*sin(\ang)})
                    arc[start angle=\ang,end angle=-\ang,radius=\R] -- cycle;
                \node[red!60!black] at ({0.55*\R},0) {\small $D\!\cap\!\Pi$};

                \draw[very thick] (0,0) -- ({\R*cos(\ang)},{\R*sin(\ang)});
                \draw[very thick] (0,0) -- ({\R*cos(-\ang)},{\R*sin(-\ang)});
                \node at ({1.25*cos(\ang)-0.22*sin(\ang)},{1.25*sin(\ang)+0.22*cos(\ang)})
                    {$\Gamma_+$};
                \node at ({1.25*cos(-\ang)-0.22*sin(-\ang)},{1.25*sin(-\ang)+0.22*cos(-\ang)})
                    {$\Gamma_-$};

                \draw[densely dashed] (0,0) circle (\R);
                \pgfmathsetmacro{\rhlabel}{180+\thext+15}
                \node at ({\R*cos(\rhlabel)+0.22*cos(\rhlabel)},{\R*sin(\rhlabel)+0.22*sin(\rhlabel)})
                    {\small $r\!=\!h$};

                \draw[very thick,blue!70!black]
                    ({\R*cos(-\thext)},{\R*sin(-\thext)})
                    arc[start angle={-\thext},end angle={\thext},radius=\R];
                \node[blue!70!black] at ({\R+0.38},0) {$\partial S^{i}$};

                \draw[very thick,red!70!black]
                    ({-\dsep*sin(\ang)},{\dsep*cos(\ang)})
                    -- ({\R*cos(\thext)},{\R*sin(\thext)});
                \draw[very thick,red!70!black]
                    ({-\dsep*sin(\ang)},{-\dsep*cos(\ang)})
                    -- ({\R*cos(-\thext)},{\R*sin(-\thext)});
                \draw[very thick,red!70!black]
                    ({-\dsep*sin(\ang)},{\dsep*cos(\ang)})
                    arc[start angle={90+\ang},end angle={270-\ang},radius=\dsep];
                \pgfmathsetmacro{\Semid}{90+\ang+0.5*(180-2*\ang)}
                \node[red!70!black] at ({(\dsep+0.35)*cos(\Semid)},{(\dsep+0.35)*sin(\Semid)})
                    {$\partial S^{e}$};

                \draw[<->] ({1.6*cos(\ang)},{1.6*sin(\ang)})
                    -- ({1.6*cos(\ang)-\dsep*sin(\ang)},
                        {1.6*sin(\ang)+\dsep*cos(\ang)});
                \node[above left] at
                    ({1.6*cos(\ang)-0.5*\dsep*sin(\ang)},
                     {1.6*sin(\ang)+0.5*\dsep*cos(\ang)})
                    {\small $\frac{1}{\tau}$};

                \draw[->] (0.55,0)
                    arc[start angle=0,end angle=\ang,radius=0.55];
                \draw[->] (0.55,0)
                    arc[start angle=0,end angle=-\ang,radius=0.55];
                \node at ({0.82*cos(0.5*\ang)},{0.82*sin(0.5*\ang)})
                    {\small $\frac{a}{2}$};
                \node at ({0.82*cos(-0.5*\ang)},{0.82*sin(-0.5*\ang)})
                    {\small $\frac{a}{2}$};

                \fill (0,0) circle (1.2pt);
                \node[below left] at (-0.05,-0.05) {$O$};
                \draw[->] (0,0) -- (2.95,0) node[below right] {$\hat{x}$};
                \draw[->] (0,0) -- (0,2.0) node[left] {$\hat{y}$};

                \node at (-1.85,1.85) {$\Pi$};
            \end{tikzpicture}
            \caption{Cross-section in the plane~$\Pi$ perpendicular to~$E$ at~$O$.
            The two sides $\Gamma_\pm$ (thick black) enclose the opening angle~$a$.
            The contour~$\partial S^e$ (red) lies at distance~$1/\tau$ from~$\Gamma_\pm$ on the exterior side,
            $\partial S^i$ (blue) is an arc on the circle $r=h$,
            and the shaded region is~$A_-$.}
            \label{fig:cross-section}
        \end{figure}
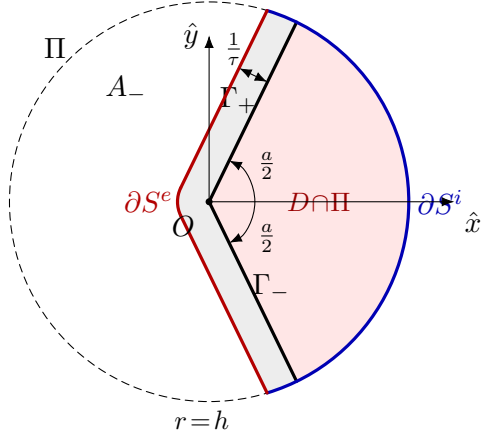

        \subsection{Integral Identity and Estimates}

        The following integral identity is the key ingredient for our stability estimate.

        \begin{proposition}\label{prop:int-identity}
            Let $u_0\in H^1_{loc}(\R^3)$ be harmonic. Then
            \begin{align}\label{eq:int-id}
                (k-1)\int_{F_{\pm}}u_0\p_\nu u d\sigma= &\int_{S^e \cup S^i\cup A_{\pm}} \left(u_0\p_\nu (u-u')-(u-u')\p_\nu u_0\right)d\sigma.
            \end{align}
        \end{proposition}
        \begin{proof}
            The proof follows by applying Green's formula in the regions $\widetilde{\O}\setminus \widetilde{D}$ and $\widetilde{D}$, together with the transmission conditions in \eqref{eq:transmission}.
        \end{proof}

        \medskip

        We now introduce complex geometrical optics (CGO) solutions to the Laplace equation. Generally speaking, CGO solutions are functions of the form $u_0(x)=e^{\zeta\cdot x}(1+\psi(x))$, where $\zeta\in \C^3$ is a complex vector satisfying $\zeta\cdot \zeta=0$. In our case, for all $\tau>0$, we define
        \begin{align}\label{eq:CGO}
            u_0(x) = e^{\zeta\cdot (x-x_O)},\quad {\rm where}\quad \zeta =\tau(-\hat{x}+i\hat{y})\in \C^3.            
        \end{align}
        Here, $x_O$ denotes the coordinates of the point $O$. The following estimates are immediate: For $x=(r,\t,z)$ in local cylindrical coordinates,
        \begin{align*}
            &\forall x\in \widetilde{\O}, \quad |u_0(x)|=e^{-\tau r\cos \t} \leq e^{\tau \cdot 1/\tau}=e,\\
            &\forall x\in \overline{\widetilde{D}}, \quad |u_0(x)|=e^{-\tau r\cos \t} \leq e^{-\tau \widetilde{a} r}\leq e^{-\tau \widetilde{a}h}.
        \end{align*}
        Here we set $\ds \widetilde{a}:=\cos(\frac{a}{2})>0$, and the corresponding estimates for $\n u_0$ can be obtained by multiplying by $\tau$.
        \medskip

        \begin{proposition}\label{prop:lowerbound}
            The left-hand side of the integral identity \eqref{eq:int-id} admits the following decomposition:
            \begin{align*}
                &\int_{F_\pm} u_0\p_\nu u d\sigma = I_{F,1} + I_{F,2,\infty}-I_{F,2,c} + I_{F,3} + I_{F,4} + I_{F,5}.
            \end{align*}
            Furthermore, for $\rho,h,h_z$ satisfying \eqref{eq:constraint1} and \eqref{eq:constraint2} and being sufficiently small, the following estimates hold for each term,
            \begin{align*}                
                &|I_{F,1}|\lesssim  \rho^{\lm_1-1}h_z\tau^{-1}, &  &|I_{F,2,\infty}| \gtrsim \rho^{\lm_1+\e_*-\mu_1} h_z \tau^{-\mu_1},\\
                &|I_{F,2,c}| \lesssim (\rho+h_z)^{1-\mu_1}\tau^{-\mu_1} e^{-\frac{\tau \widetilde{a} h}{2}}, &  &|I_{F,3}| \lesssim \rho^{-(s-\mu_1-1/2)}h_z \tau^{-(s-1/2)},\\
                &|I_{F,4}| \lesssim \rho^{-\widetilde{\mu_2}}h_z\tau^{-\widetilde{\mu_2}},&  &|I_{F,5}| \lesssim h_z\tau^{-1}.
            \end{align*}
            Here, $\e_*>0$ is a small constant, $\widetilde{\mu_2}=\mu_2$ if $s-\mu_2\geq 1/2$ and $\widetilde{\mu_2}=s-1/2$ if $s-\mu_2< 1/2$.
        \end{proposition}
        \begin{proof}
            We split $u$ into singular and regular parts as stated in \zcref{th:SingDecom}. We mainly focus on the first edge singular function. More precisely, we decompose
            \begin{align*}
                u=u_{\bf v}+v_1+\widetilde{v_1}+u_{E,2}+u_{\rm reg},
            \end{align*}
            where 
            \begin{itemize}
                \item $u_{\bf v}$ is the vertex singular function associated with $\bf v$,
                \item $\ds v_1(\widetilde{r},\t,z):=\g_{E,\mu_1}(z)\widetilde{r}^{\mu_1}\vf_{E,\mu_1}(\t)$,
                \item $\ds \widetilde{v_1}=(\mathcal{K}[\g_{E,\mu_1}](\widetilde{r},z)-\g_{E,\mu_1}(z))\widetilde{r}^{\mu_1}\vf_{E,\mu_1}(\t)$,
                \item $u_{E,2}$ is the remaining edge singular functions with the leading singular exponent $\mu_2$,
                \item $u_{\rm reg}$ is the regular part.
            \end{itemize}
            We also extend the rays $\Gamma_\pm$ to infinity and denote the corresponding faces by $F_\pm^\infty$. Straightforward computations give
            \begin{align*}
                \int_{F_\pm} u_0\p_\nu u d\sigma = I_{F,1} + I_{F,2,\infty}-I_{F,2,c} + I_{F,3} + I_{F,4} + I_{F,5}.
            \end{align*}
            where
            \begin{align*}
                &I_{F,1} := \int_{F_{\pm}}u_0\p_\nu u_{\bf v} d\sigma, & & I_{F,2,\infty} := \int_{F_{\pm}^\infty}u_0\p_\nu v_1 d\sigma,\\
                &I_{F,2,c} := \int_{F_{\pm}^\infty\setminus F_{\pm}}u_0\p_\nu v_1 d\sigma,& & I_{F,3} := \int_{F_{\pm}}u_0\p_\nu \widetilde{v_1} d\sigma,\\
                &I_{F,4} := \int_{F_{\pm}}u_0\p_\nu u_{E,2} d\sigma, & & I_{F,5} := \int_{F_{\pm}}u_0\p_\nu u_{\rm reg} d\sigma.
            \end{align*}
            
            We first focus on the integral $I_{F,2,\infty}$. Since the variables are separated, we have
            \begin{align*}
                &\int_{F^\infty_\pm}u_0\p_\nu v_1 d\sigma = \int_0^{h_z} \int_0^{+\infty} u_0(r,\pm \frac{a}{2},z)\p_\nu v_1(r,\pm \frac{a}{2},z) dr dz,\\
                &= \int_0^{h_z} \g_{E,\mu_1}(z) \d(z)^{-\mu_1}dz \left( \vf'_{E,\mu_1}(\frac{a}{2}) \int_0^{+\infty}  r^{\mu_1-1} e^{z^+_a r} dr- \vf'_{E,\mu_1}(-\frac{a}{2}) \int_0^{+\infty}  r^{\mu_1-1} e^{z^-_a r} dr \right).
            \end{align*}
            By construction of the CGO solution, $u_0(r,\pm \frac{a}{2},z) = e^{z^\pm_a r}$ with $z^\pm_a=-\tau e^{\mp i\frac{a}{2}}$. To evaluate the integral in $r$, we use a complex-analysis argument from Proposition 4.4 in \cite{moi5}. Thus,
            \begin{align*}
                \int_0^{+\infty}  r^{\mu_1-1} e^{z^{\pm}_a r} dr =\Gamma(\mu_1)(\frac{-1}{z^\pm_a})^{\mu_1};
            \end{align*}
            Here $\Gamma$ denotes Euler's Gamma function. Hence,
            \begin{align*}
                \left| \vf'_{E,\mu_1}(\frac{a}{2}) \int_0^{+\infty}  r^{\mu_1-1} e^{z^+_a r} dr- \vf'_{E,\mu_1}(-\frac{a}{2}) \int_0^{+\infty}  r^{\mu_1-1} e^{z^-_a r} dr \right| \geq \sin(a\mu_1)\Gamma(\mu_1)\tau^{-\mu_1}.
            \end{align*}
            We now estimate the integral with respect to $z$, $\ds \int_0^{h_z} \g_{E,\mu_1}(z) \d(z)^{-\mu_1}dz$. We use the assumptions 3.) and 4.) in \zcref{assumption:singularity} and the asymptotic behavior \eqref{eq:asym-coef} to derive that $\ds \g_{E,\mu_1}(\rho) \sim a_{E,\mu_1}^{{\bf v},\lm_k} \rho^{\lm_k}$ as $\rho \to 0$. This asymptotic behavior implies that $\g_{E,\mu_1}(\rho)$ keeps the same sign as $\rho \rightarrow 0$, and we will assume that it is positive without loss of generality. By choosing $\e_*>0$, we have,
            \begin{align*}
                \frac{\g_{E,\mu_1}(\rho)}{\rho^{\lm_1+\e_*}} \tendvers{\rho}{0} +\infty.
            \end{align*}
            Consequently, there exists a constant $C_{\e_*}>0$ such that 
            \begin{align*}
                \forall 0<\rho<C_{\e_*},\quad \g_{E,\mu_1}(\rho) \geq  \rho^{\lm_1+\e_*}.
            \end{align*} 
            We remark that $C_{\e_*}$ is a local constant, depending on the inclusion $D$, the edge $E$, the vertex $\bf v$, and the parameter $\e_*$.

            We also remark that the function $\d$ is equivalent to the distance between the point under consideration and the vertex $\bf{v}$. Then there exist constants $C_\pm>0$, which depends only on the geometry and thus considered as an {\it a priori} parameter, such that for all $z\in (0,h_z)$,
            \begin{align}\label{eq:delta-eqv}
                C_- (\rho+z) \leq &\d(z)\leq C_+ (\rho+z).
            \end{align}
            We assume hereafter that 
            \begin{align}\label{eq:constraint3}
                \rho+h_z<C_{\e_*},
            \end{align}     
            Hence,   
            \begin{align*}
                \left|\int_0^{h_z} \g_{E,\mu_1}(z) \d(z)^{-\mu_1} dz \right| \geq \frac{C_+^{-\mu_1}}{\lm_1+\e_*-\mu_1+1} \left((\rho+h_z)^{\lm_1+\e_*-\mu_1+1}-\rho^{\lm_1+\e_*-\mu_1+1}\right). 
            \end{align*}           
            We may choose $\e_*$ smaller than a certain {\it a priori} constant, so that
            \begin{align}\label{eq:esti-F2-infty}
                |I_{F,2,\infty}|=\left| \int_{F_\pm^\infty} u_0 \p_\nu v_1 d\sigma \right| \gtrsim  h_z \rho^{\lm_1+\e_*-\mu_1} \tau^{-\mu_1}.
            \end{align}

            Next, we consider the integral $I_{F,2,c}$ on the complementary faces. For $s,x >0$, we introduce the incomplete Gamma function
            \begin{align*}
                \Gamma(s,x)=\int_x^{+\infty} t^{s-1} e^{-t} dt.
            \end{align*}
            Using the inequality $\ds e^{-t}\leq e^{-t/2}e^{-x/2}$ for $t\geq x$, we deduce that
            \begin{align*}
                \Gamma(s,x)\leq 2^s \Gamma(x) e^{-x/2}.
            \end{align*}
            Then, by a simple change of variables, we obtain
            \begin{align*}
                \left|\int_h^{+\infty} r^{\mu_1-1} e^{z_a^\pm r} dr \right|& \leq \int_h^{+\infty} r^{\mu_1-1} e^{-\tau \widetilde{a} r} dr\leq  (\tau \widetilde{a})^{-\mu_1} \Gamma(\mu_1, \tau \widetilde{a} h)\\
                &\leq (\tau \widetilde{a})^{-\mu_1} 2^{\mu_1} \Gamma(\mu_1) e^{-\frac{\tau \widetilde{a} h}{2}}.
            \end{align*}
            Applying the estimates of singular coefficients \eqref{eq:esti-edge-coef}, we deduce that
            \begin{align*}
                \left|\int_0^{h_z} \g_{E,\mu_1}\d(z)^{-\mu_1}dz \right|&\leq C_+^{-\mu_1}\|\g_{E,\mu_1}\|_{L^2(E)} \frac{(\rho+h_z)^{1-\mu_1}-\rho^{1-\mu_1}}{1-\mu_1}\\
                &\leq C_+^{-\mu_1}C\|u\|_{H^1(\Omega)}\frac{(\rho+h_z)^{1-\mu_1}}{1-\mu_1}.
            \end{align*}
            Consequently,
            \begin{align}\label{eq:esti-F2-c}
                |I_{F,2,c}|&\leq 2\mu_1 (\tau \widetilde{a})^{-\mu_1} 2^{\mu_1} \Gamma(\mu_1) e^{-\frac{\tau \widetilde{a} h}{2}} C_+^{-\mu_1}C\|u\|_{H^1(\Omega)}\frac{(\rho+h_z)^{1-\mu_1}}{1-\mu_1}\nonumber\\
                &\lesssim \tau^{-\mu_1}(\rho+h_z)^{1-\mu_1}e^{-\frac{\tau \widetilde{a} h}{2}}.
            \end{align}

            Applying \zcref{lemma:K-gamma} and the estimates of singular coefficients \eqref{eq:esti-edge-coef}, we obtain
            \begin{align}\label{eq:esti-F3}
                |I_{F,3}|&\leq \int_0^{h_z}\int_0^h |\mathcal{K}[\g_{E,\mu_1}](r,z)-\g_{E,\mu_1}(z)| r^{\mu_1-1} e^{-\tau \widetilde{a} r} dr dz \nonumber\\
                & \leq C'_{\phi,\mu_1} \|\g_{E,\mu_1}\|_{V_{-s}^{s-\mu_1}(E)}\rho^{-(s-\mu_1-1/2)} \int_0^{h_z}\int_0^h r^{s-\mu_1-1/2+\mu_1-1} e^{-\tau \widetilde{a} r} dr dz \nonumber\\
                &\lesssim h_z \rho^{-(s-\mu_1-1/2)} \int_0^{+\infty} r^{s-3/2} e^{-\tau\widetilde{a} r} dr \nonumber\\
                & \lesssim h_z \rho^{-(s-\mu_1-1/2)} (\tau\widetilde{a})^{-(s-1/2)}\Gamma(s-1/2)\lesssim h_z \rho^{-(s-\mu_1-1/2)} \tau^{-(s-1/2)}.
            \end{align}

            The estimates of $I_{F,1}$, $I_{F,4}$, and $I_{F,5}$ are relatively straightforward. Consider a point $x=(r,\pm a/2, z)\in F_\pm$.
            \begin{itemize}
                \item It readily follows from the vertex singular functions and the estimates of singular coefficients that
            \begin{align*}
                |\n u_{\bf v}(x)| & \lesssim \rho^{\lm_1-1}.
            \end{align*}
            We remark that $\lm_1>1$.
            \item For the remaining edge singular functions, we apply \zcref{lemma:esti-grad-edge} and \eqref{eq:esti-edge-coef} to obtain
            \begin{align*}
                |\n u_{E,2}(x)| & \lesssim \frac{r^{\widetilde{\mu_2}-1}}{\rho^{\widetilde{\mu_2}}},
            \end{align*}
            where $\widetilde{\mu_2}=\mu_2$ if $s-\mu_2\geq 1/2$ and $\widetilde{\mu_2}=s-1/2$ if $s-\mu_2< 1/2$.
            \item Combining the Sobolev embedding $H^{3/2}\hookrightarrow L^\infty$ in $\R^3$ with \eqref{eq:esti-reg-cinqdemi}, we deduce that
            \begin{align}\label{eq:esti-reg-infinity}
                \|\n u_{\rm reg}\|_{L^\infty(\Omega)} \leq C \|u_{\rm reg}\|_{H^{5/2}(\Omega)} \lesssim 1.
            \end{align}
            \end{itemize} 

            Substituting these estimates into the integrals, we obtain
            \begin{align}
                |I_{F,1}| &\lesssim \rho^{\lm_1-1} \int_0^{h_z}\int_0^h e^{-\tau\widetilde{a} r} dr dz \lesssim \rho^{\lm_1-1}h_z \frac{1-e^{-\tau\widetilde{a} h}}{\tau \widetilde{a}} \lesssim \rho^{\lm_1-1}h_z \tau^{-1} ,\label{eq:esti-F1}\\
                |I_{F,4}| &\lesssim \rho^{-\widetilde{\mu_2}} \int_0^{h_z}\int_0^h r^{\widetilde{\mu_2}-1} e^{-\tau\widetilde{a} r} dr dz\lesssim \rho^{-\widetilde{\mu_2}} h_z \int_0^{+\infty} r^{\widetilde{\mu_2}-1} e^{-\tau\widetilde{a} r} dr  \nonumber \\
                &\lesssim \rho^{-\widetilde{\mu_2}}h_z (\tau\widetilde{a})^{-\widetilde{\mu_2}}\Gamma(\widetilde{\mu_2})\lesssim \rho^{-\widetilde{\mu_2}}h_z \tau^{-\widetilde{\mu_2}},\label{eq:esti-F4}\\
                |I_{F,5}| &\lesssim h_z \frac{1-e^{-\tau\widetilde{a} h}}{\tau \widetilde{a}} \lesssim h_z \tau^{-1}.\label{eq:esti-F5}
            \end{align}

            The claim follows by collecting the estimates \eqref{eq:esti-F2-infty}, \eqref{eq:esti-F2-c}, \eqref{eq:esti-F3}, \eqref{eq:esti-F1}, \eqref{eq:esti-F4} and \eqref{eq:esti-F5}.
        \end{proof}

        \medskip

        \begin{proposition}\label{prop:upperbound}
            For $\tau>0$ large enough, the following estimate holds,
            \begin{align*}
                &\left|\int_{S^e \cup S^i\cup A_{\pm}} \left(u_0\p_\nu (u-u')-(u-u')\p_\nu u_0\right)d\sigma\right|\\
                &\lesssim h^{-1}\tau^{-2}+\rho^{-\mu_1}\tau^{-(\mu_1+1)}+h_z e^{-\tau \widetilde{b} h} (h^{\mu_1}\rho^{-\mu_1}+1+h\tau)\\
                &\quad +h h_z \left(\|\n(u-u')\|_{L^\infty(S^e)}+\tau\|u-u'\|_{L^\infty(S^e)}\right).
            \end{align*}
        \end{proposition}
        \begin{proof}
            First, we estimate the integrals on the upper and lower faces $A_\pm$. The normal vectors on these faces are parallel to the $z$-axis. Therefore, $\p_\nu u_0=\p_z u_0=0$ on these regions, and hence
            \begin{align*}
                \int_{A_\pm} (u-u')\p_\nu u_0 d\sigma =0.
            \end{align*}
            We now estimate $\p_z u$ and $\p_z u'$ on $A_\pm$. Since the region $\widetilde{\O}$ is at distance at least $h$ from the polyhedron $D'$, interior elliptic regularity implies that $u'$ is analytic there. Applying the gradient estimate (e.g. Theorem 3.9 in \cite{Gilbarg-Trudinger}), we have
            \begin{align}\label{eq:esti-grad-uprime}
                \|\n u'\|_{L^\infty(\widetilde{\O})} \leq C h^{-1} \|u'\|_{L^\infty(\O)}\lesssim h^{-1}.
            \end{align}
            For $\n u$, we must account for singularities. We decompose $u$ into singular and regular parts by \zcref{th:SingDecom}, then estimate each term.
            \begin{align*}
                u=u_{\bf v}+u_{E} +u_{\rm reg}.
            \end{align*}

            Since singular exponents are ordered increasingly, it is sufficient to consider the strongest singularity. For vertex singular functions, it is straightforward that $\ds |\p_z u_{\bf v}(x)| \leq C \rho^{\lm_1-1}$. For edge singular functions, combining \zcref{lemma:esti-grad-edge} with \eqref{eq:esti-edge-coef}, we have
            \begin{align*}
                |\p_z u_{E}(x)| \leq C \frac{\widetilde{r}^{\mu_1-1}}{\d(z)}\lesssim \rho^{-\mu_1}r^{\mu_1-1}.
            \end{align*}
            For the regular part, the estimate comes directly from \eqref{eq:esti-reg-infinity}. Combining those estimates, we have,
            \begin{align*}
                \|\p_z (u-u')\|_{L^\infty(\widetilde{\O})} \lesssim (h^{-1} +\rho^{-\mu_1}r^{\mu_1-1}).
            \end{align*}
            
            Calculating the integral over $A_\pm^i$,
            \begin{align*}
                \int_{A_\pm^i}|u_0(x)| r^{\mu_1-1}d\sigma &= \int_{-a/2}^{a/2}\int_0^h e^{-\tau r\cos \t} r^{\mu_1} dr d\t \leq  \int_0^h\int_{-a/2}^{a/2} e^{-\tau \widetilde{a} r} r^{\mu_1} dr d\t \\
                &\leq a \int_0^{+\infty} e^{-\tau \widetilde{a} r} r^{\mu_1} dr \leq C \tau ^{-(\mu_1+1)}\Gamma(\mu_1+1).
            \end{align*}
            So, 
            \begin{align}\label{eq:bound-Ai}
                \left|\int_{A_\pm^i} u_0 \p_\nu(u-u')d\sigma \right|\lesssim (h^{-1}\tau^{-2}+\rho^{-\mu_1}\tau^{-(\mu_1+1)}).
            \end{align}

            To estimate the integrals on $A_\pm^e$, we split $A_\pm^e$ into two parts. Denote by $B_{\tau}$ the two-dimensional disk on $\Pi$, centered at $O$ with radius $L\tau^{-1}$. Here, $L>1$ is chosen such that $\arcsin(1/L)<(\pi-a)/4$. With this construction, we deduce that
            \begin{align*}
                (r,\t)\in A_-^e\setminus B_{\tau} \Rightarrow L/\tau <r< h \;\; {\rm and}\;\; \t\in (a/2,(\pi+a)/4)\cup (-(\pi+a)/4,-a/2).
            \end{align*}
            Introducing $\widetilde{\b}:=\cos(\frac{\pi+a}{4})>0$, we have $u_0(r,\t,z)\leq e^{-\tau \widetilde{\b} r}$ on $A_-^e\setminus B_{\tau}$. We now compute
            \begin{align*}
                \int_{A_-^e\cap B_{\tau}} |u_0(x)|r^{\mu_1-1} d\sigma &\leq e\int_{B_{\tau}} r^{\mu_1-1} d\sigma= 2\pi e \int_0^{L\tau^{-1}} r^{\mu_1} dr = \frac{2\pi e}{\mu_1+1} (L\tau^{-1})^{\mu_1+1}\\
                &\leq C\tau^{-(\mu_1+1)},\\
                \int_{A_-^e\setminus B_{\tau}} |u_0(x)|r^{\mu_1-1} d\sigma &\leq 2\int_{a/2}^{(\pi+a)/4} \int_{L\tau^{-1}}^h r^{\mu_1} e^{-\tau \widetilde{\b} r} dr d\t \leq C \int_{0}^{+\infty} r^{\mu_1} e^{-\tau \widetilde{\b} r} dr \\
                &\leq C \tau^{-(\mu_1+1)}\Gamma(\mu_1+1).
            \end{align*}
            Hence, it follows that
            \begin{align}\label{eq:bound-Ae}
                \left|\int_{A_\pm^e} u_0 \p_\nu (u-u')d\sigma\right|\lesssim h^{-1}\tau^{-2}+\rho^{-\mu_1}\tau^{-(\mu_1+1)}.
            \end{align}


            Next, we estimate the integrals on the inner lateral face $S^i$. One readily checks that $|u_0(x)|\leq e^{-\tau\widetilde{b}h}$ on $S^i$, with $\widetilde{b}=\cos((\pi+a)/4)>0$, once the following condition is satisfied:
            \begin{align}\label{eq:tau0}
                \tau\geq \tau_0=\frac{1}{h\sin((\pi-a)/4)}.
            \end{align}
            We assume this condition throughout the remainder of the proof. Then,
            \begin{align*}
                &\left|\int_{S^i} u_0 \p_\nu (u-u')-(u-u')\p_\nu u_0 d\sigma\right| \leq \left(\|\p_\nu(u-u')\|_{L^\infty(S^i)} + \tau\|u-u'\|_{L^\infty(S^i)}\right)|S^i| e^{-\tau \widetilde{b} h}.
            \end{align*}
            The quantity $\|u-u'\|_{L^\infty(S^i)}$ can be estimated by H\"older norms that depend only on {\it a priori} data. For $\|\p_\nu(u-u')\|_{L^\infty(S^i)}$, we apply \zcref{lemma:esti-grad-edge} to obtain $\|\p_\nu u\|_{L^\infty(S^i)}\lesssim h^{\mu_1-1}\rho^{-\mu_1}$, and we use \eqref{eq:esti-grad-uprime} for $\n u'$. Consequently, we obtain the following estimate for the integrals on $S^i$:
            \begin{align}\label{eq:bound-Si}
                \left|\int_{S^i} u_0 \p_\nu (u-u')-(u-u')\p_\nu u_0 d\sigma\right|\lesssim hh_z e^{-\tau \widetilde{b} h} (h^{\mu_1-1}\rho^{-\mu_1}+h^{-1}+\tau).
            \end{align}

            Finally, we estimate the integrals on the lateral face $S^e$. Since we use propagation of smallness to estimate $u-u'$ and $\n(u-u')$ on $S^e$, the following bounds are obtained by straightforward calculations:
            \begin{align}\label{eq:bound-Se}
                &\left|\int_{S^e} u_0 \p_\nu (u-u')- (u-u')\p_\nu u_0 d\sigma\right| \nonumber \\
                 &\leq |S^e|\left(\|u_0\|_{L^\infty(S^e)}\|\n (u-u')\|_{L^\infty(S^e)} +\|\n u_0\|_{L^\infty(S^e)}\|(u-u')\|_{L^\infty(S^e)}\right)  \nonumber \\
                &\lesssim h h_z \left(\|\n(u-u')\|_{L^\infty(S^e)}+\tau\|u-u'\|_{L^\infty(S^e)}\right).
            \end{align}
            The proof is complete by collecting the results \eqref{eq:bound-Ai}, \eqref{eq:bound-Ae}, \eqref{eq:bound-Si}, and \eqref{eq:bound-Se}.
        \end{proof}

        \section{Proof of the stability}\label{sec:proof}

        \begin{proof}[Proof of \zcref{th:main}]
            We apply the integral identity \eqref{eq:int-id} with the CGO solution $u_0$ constructed in \eqnref{eq:CGO}. The left-hand side of \eqref{eq:int-id} is estimated using \zcref{prop:lowerbound}, and the right-hand side is estimated using \zcref{prop:upperbound}. Collecting these estimates, we obtain
            \begin{align*}
                \rho^{\lm_1+\e_*-\mu_1} h_z \tau^{-\mu_1} &\lesssim \rho^{\lm_1-1}h_z\tau^{-1} + (\rho+h_z)^{1-\mu_1}\tau^{-\mu_1} e^{-\frac{\tau \widetilde{a} h}{2}}\\
                & \quad + \rho^{-(s-\mu_1-1/2)}h_z \tau^{-(s-1/2)} + \rho^{-\widetilde{\mu_2}}h_z\tau^{-\widetilde{\mu_2}} + h_z\tau^{-1}\\
                & \quad + h^{-1}\tau^{-2}+\rho^{-\mu_1}\tau^{-(\mu_1+1)}+h_z e^{-\tau \widetilde{b} h} (h^{\mu_1}\rho^{-\mu_1}+1+h\tau)\\
                &\quad +h h_z \left(\|\n(u-u')\|_{L^\infty(S^e)}+\tau\|u-u'\|_{L^\infty(S^e)}\right).
            \end{align*}
            To estimate $\|u-u'\|_{L^\infty(S^e)}$ and $\|\n(u-u')\|_{L^\infty(S^e)}$, we apply propagation of smallness from \zcref{prop:propa-small}. Introduce the notation
            \begin{align*}
                \zeta(\e):=(\ln |\ln \e|)^{-\a}.
            \end{align*}
            Noting that the common H\"older norm $T$ depends only on {\it a priori} data and that the surface $S^e$ has a distance $1/\tau$ to $D$, we have,
            \begin{align*}
                \|u-u'\|_{L^\infty(S^e)} &\lesssim \zeta(\e),\\
                \|\n(u-u')\|_{L^\infty(S^e)} &\lesssim \tau \zeta(\e).
            \end{align*}
            Recall that $\rho,h$, and $h_z$ must satisfy the constraints \eqref{eq:constraint1},\eqref{eq:constraint2} and \eqref{eq:constraint3}, also that $D'$ is at distance at least $h$ from $\widetilde{\O}$. Thus, we choose
            \begin{align*}
                h_z=h,\quad \rho=\frac{h}{\ell},\quad h=\min\left\lbrace C_{\e_*}\left(1+\frac{1}{\ell}\right)^{-1} ,\frac{\ell^2}{2\ell+1},\;\frac{1}{2}d_H(D,D')\right\rbrace.
            \end{align*}
            Then the estimate becomes
            \begin{align*}
                h^{\lm_1+\e_*-\mu_1+1} \tau^{-\mu_1} &\lesssim h^{\lm_1}\tau^{-1} + h^{1-\mu_1}\tau^{-\mu_1} e^{-\frac{\tau \widetilde{a} h}{2}} + h^{3/2-s+\mu_1} \tau^{-(s-1/2)} + h^{1-\widetilde{\mu_2}}\tau^{-\widetilde{\mu_2}} + h\tau^{-1}\\
                & \quad + h^{-1}\tau^{-2}+h^{-\mu_1}\tau^{-\mu_1-1}+h e^{-\tau \widetilde{b} h} (1+h\tau) +h^2 \tau \zeta(\e).
            \end{align*}
            Next, we use the usual inequalities $e^{-x}\leq x^{-1}$ and $e^{-x}\leq x^{-2}$ for all $x>0$ to deduce that
            \begin{align*}
                h^{\lm_1+\e_*-\mu_1+1} \tau^{-\mu_1} &\lesssim h^{\lm_1}\tau^{-1} + h^{-\mu_1}\tau^{-\mu_1-1} + h^{3/2-s+\mu_1} \tau^{-(s-1/2)} + h^{1-\widetilde{\mu_2}}\tau^{-\widetilde{\mu_2}} + h\tau^{-1}\\
                & \quad + h^{-1}\tau^{-2}+h^{-\mu_1}\tau^{-\mu_1-1}+\tau^{-1} +h^2 \tau \zeta(\e).
            \end{align*}
            Recall the estimates on singular exponents: $0<\lm_1<s-1/2$, $0<\mu_1<1$, and $3/2\leq s <\mu_1+1$. Combining these with $\tau\geq 1$ and $h\leq 1$, we deduce that
            \begin{align*}
                h^{\lm_1+\e_*-\mu_1+1} \tau^{-\mu_1} &\lesssim h^{-1}\tau^{-1}+  h^{1-\widetilde{\mu_2}}\tau^{-\widetilde{\mu_2}}+h^2 \tau \zeta(\e).
            \end{align*}
            Recalling that $s=3/2+\epsilon$ with $0<\epsilon<\mu_1-1/2$ and using the definition of $\widetilde{\mu_2}$, we distinguish the following two cases.
            \begin{itemize}
                \item If $s-\mu_2\geq 1/2$, then $-\widetilde{\mu_2}=-\mu_2$ and $1-\widetilde{\mu_2}=1-\mu_2\geq 3/2-s=\epsilon>-1/2$.
                \item If $s-\mu_2< 1/2$, then $-\widetilde{\mu_2}=1/2-s=-1-\epsilon <-1$ and $1-\widetilde{\mu_2}=-\epsilon >-1/2$.
            \end{itemize}
            Taking into account that $\mu_2>1$, the estimate becomes
            \begin{align*}
                h^{\lm_1+\e_*-\mu_1+1} \tau^{-\mu_1} &\lesssim h^{-1}\tau^{-1}+h^2 \tau \zeta(\e).
            \end{align*}
            By minimizing the right-hand side with respect to $\tau$, we choose
            \begin{align}
                \tau=\tau_e:=\left(h^3\zeta(\e)\right)^{-\frac{1}{2}}.
            \end{align}
            For $\e$ sufficiently small, one can verify that $\tau_e\geq \tau_0$, where $\tau_0$ is given in \eqref{eq:tau0}.

            Substituting this choice into the estimate and solving for $h$, we obtain
            \begin{align*}
                h \lesssim \zeta(\e)^{\frac{1}{2(\lm_1+\e_*)+\mu_1+1}}\leq \zeta(\e)^{\frac{1}{3\lm_1+\mu_1+1}},
            \end{align*}
            where the last inequality holds by choosing $\e_*<\lm_1/2$.

            Finally, as $\ds \zeta(\e)\tendvers{\e}{0} 0$, we can deduce that for $\e$ small enough, the right-hand side above is less than both $C_{\e_*}\left(1+\frac{1}{\ell}\right)^{-1}$ and $\frac{\ell^2}{2\ell+1}$. Hence, the stability estimate \eqref{eq:stable} follows with the exponent $\kappa$ given by
            \begin{align*}
               \kappa:= \frac{-\a }{3\lm_1+\mu_1+1}.
            \end{align*}
            The exponent $\kappa$ is bounded below by a constant depending only on {\it a priori} data, using $\a=\min\{\a_0,\epsilon,\lm_1,\mu_1\}$ from \zcref{prop:propa-small} and the bounds $\mu_1\in (\frac{1}{2},1)$, $\lm_1\in (1,\sqrt{2k})$.

            The proof is complete.
        \end{proof}

        \medskip
  
        \begin{remark}
            The assumption 4.) in \zcref{assumption:singularity} can be relaxed to certain order $k\in\N$ such that the interaction coefficient $a_{E,\nu}^{{\bf v},\lm_k}$ is non-zero from this rank. In that setting, it is sufficient to replace $\lm_1$ by $\lm_k$ accordingly in the proofs in this paper. The stability remains valid with a weaker exponent $\kappa$ in \eqref{eq:stable}. 
        \end{remark}

        \medskip

        \begin{remark}
            By letting $\e\rightarrow 0$, we can deduce the uniqueness result for the inverse conductivity problem for convex polyhedral inclusions. As a significant improvement over the previous result \cite{Barcelo1994}, the assumption on the boundary data can be relaxed to a generic condition as described here.
        \end{remark}        
    
        \appendix

        \section{Edge singularity and smoothing operator}\label{explain:edge}
        Edge singular functions are constructed with an exponent $\mu$ and a coefficient $\g$ in the following form (without $\log$ terms, for simplicity):
        \begin{align}\label{eq:edge-singular}
            u_{E,\mu}(r,\t,z)= \mathcal{K}[\g](\widetilde{r},z)\widetilde{r}^\mu\vf(\t).
        \end{align}
        In full generality, $\mu\in\C$ comes from the spectrum of the operator pencil associated with the edge $E$. In our case, $\mu$ are roots of \eqref{eq:edge-expo} and can be ordered increasingly as $0<\mu_1<\mu_2<\cdots$. We also have $\frac{1}{2}<\mu_1<1$.

        Essentially, edge singularities behave as vertex singularities on cross-sections perpendicular to the edge, except that singular coefficients $\g(z)$ are functions along the edge rather than constants. They are obtained by inverting the operator pencil and applying an inverse Mellin transform. Since this approach does not directly provide the required regularity, we introduce weighted Sobolev spaces $V_{\eta}^m(E)$ and the smoothing operator $\mathcal{K}$.

        First, we introduce a stretching change of variables. Let $\d:\overline{E}\rightarrow \R$ be a positive smooth function equivalent to the distance to the endpoints of $E$. For instance, we can choose $\d(z)=1-z^2$ if $E=(-1,1)$. Let $z_0$ be an interior point of $E$; we define the stretching variable $\widetilde{z}$ by
        \begin{align*} 
            \widetilde{z}:=\int_{z_0}^z \frac{1}{\d(s)}ds.
        \end{align*}
        For a function $f$ defined on $E$, define the stretched function $\widetilde{f}$ by $\widetilde{f}(\widetilde{z}):=f(z)$. As an immediate consequence of this change of variables, we have
        \begin{align*}
            &\p_{\widetilde{z}} \widetilde{f}(\widetilde{z})=\d(z)\p_z f(z), &      &\int_\R \widetilde{f}(\widetilde{z})d\widetilde{z}=\int_E \d(z)^{-1}f(z) dz.
        \end{align*}
        Then, for $p\in [1,+\infty)$ and $k\in\N$, we deduce that $\p^k_{\widetilde{z}}\widetilde{f}\in L^p(\R)$ if and only if $\d^{k-1/p}\p_z^k f\in L^p(E)$. This motivates the following weighted Sobolev spaces.
        
        Let $m\in\N$ and $\eta\in\R$,
        \begin{align*}
            V_{\eta}^m(E):=\setdef{f\in L^2(E)}{\d^{\eta+k}\p_z^k f \in L^2(E), \; \forall k\leq m}.
        \end{align*}
        If $m$ is non-integer, $V_{\eta}^m(E)$ is defined by interpolation. From the previous result, $V^m_{-1/2}(E)$ is isomorphic to $H^m(\R)$ under the stretching change of variables. Since $\d$ is positive and smooth on $E$, the spaces $V_{\eta}^m(E)$ are increasing, i.e., there are continuous embeddings $V_{\eta_1}^m(E)\hookrightarrow V_{\eta_2}^m(E)$ for $\eta_1\leq \eta_2$. In our case, the edge-singularity coefficient satisfies $\g\in V_{-s}^{s-\mu}(E)$, where $s_0\leq s <s_0+1$ and $s_0$ is given by \eqref{eq:def-s0}. With our choice of $s_0$, we always have $\frac{3}{2}\leq s < \mu_1+1$. Therefore,
        \begin{align}\label{eq:embedding-weighted}
                V_{-s}^{s-\mu}(E)\hookrightarrow V_{-1/2}^{s-\mu}(E) \cong H^{s-\mu}(\R).
        \end{align}
         It is worth noting the role of $\mu$. Since $s\geq 3/2$, we have $s-\mu> 1/2$ for all $\mu < 1$, in particular for $\mu=\mu_1$. Using the one-dimensional Sobolev embedding $H^{s-\mu}\hookrightarrow L^\infty$, we deduce that $\widetilde{\g}\in L^\infty(\R)\Rightarrow\g\in L^\infty(E)$ in this case. For further details on these Kondratiev-type weighted Sobolev spaces, we refer to \cite{Mazya2010}.

        Next, we introduce the smoothing operator $\mathcal{K}$ by the following convolution product,
        \begin{align*}
            \forall \widetilde{r}=\frac{r}{\d(z)}>0,\; z\in E,\quad \mathcal{K}[\g](\widetilde{r},z):=\int_{\R} \frac{1}{\widetilde{r}}\phi(\frac{t}{\widetilde{r}})\widetilde{\g}(\widetilde{z}-t)dt.
        \end{align*}
        Here, the kernel $\phi\in\Sw(\R)$ is smooth, rapidly decreasing and satisfies $\ds \int_\R \phi(t) dt=1$.

        \medskip

        \begin{lemma}\label{lemma:esti-grad-edge}
            Let $\frac{1}{2}<\mu<s$ and let $u_{E,\mu}$ be given by \eqref{eq:edge-singular} with coefficient $\g\in V_{-s}^{s-\mu}(E)$. Then, in the region $\widetilde{\O}$, there exists a constant $C_{\phi,\mu}>0$ such that
            \begin{align}\label{eq:esti-grad-edge}
                |\n u_{E,\mu}(r,\t,z)| \leq C_{\phi,\mu} \|\g\|_{V_{-s}^{s-\mu}(E)}\frac{r^{\widetilde{\mu}-1}}{\rho^{\widetilde{\mu}}}.
            \end{align}
            Here, $\rho$ denotes the distance from the point $(r,\t,z)$ to its closest endpoint of $E$. Also, $\widetilde{\mu}=\mu$ if $s-\mu\geq 1/2$ and $\widetilde{\mu}=s-1/2$ if $s-\mu<1/2$.
        \end{lemma}
        \begin{proof}
            By a change of variables stretching $E$ into $\R$ and the embedding \eqref{eq:embedding-weighted}, we have $\widetilde{\g}\in H^{s-\mu}(\R)$. We compute $\p_z u_{E,\mu}$ via the chain rule:
            \begin{align*}
                \frac{\p}{\p z}u_{E,\mu}(r,\t,z)&= \left(\frac{\p \widetilde{z}}{\p z}\frac{\p}{\p \widetilde{z}}+\frac{\p \widetilde{r}}{\p z}\frac{\p}{\p \widetilde{r}}\right) u_{E,\mu}(r,\t,z)\\
                &= \frac{1}{\d(z)}\left(\frac{\p}{\p \widetilde{z}}\mathcal{K}[\g]\right)\widetilde{r}^{\mu}\vf(\t)-\frac{\d'(z)}{\d(z)} \left(\widetilde{r}^{\mu+1} \frac{\p}{\p \widetilde{r}}\mathcal{K}[\g]+\mu \widetilde{r}^{\mu}\mathcal{K}[\g]\right) \vf(\t)\\
                \frac{\p}{\p r} u_{E,\mu}(r,\t,z) &= \frac{\p \widetilde{r}}{\p r}\frac{\p}{\p \widetilde{r}} u_{E,\mu}(r,\t,z)=\frac{1}{\d(z)}\left(\frac{\p}{\p \widetilde{r}}\mathcal{K}[\g]\right)\widetilde{r}^{\mu}\vf(\t)+\frac{\mu}{\d(z)} \mathcal{K}[\g]\widetilde{r}^{\mu-1}\vf(\t),\\
                \frac{1}{r} \frac{\p}{\p \t} u_{E,\mu}(r,\t,z) &= \frac{1}{\d(z)}\mathcal{K}[\g]\widetilde{r}^{\mu-1}\vf'(\t).
            \end{align*}
            By the convolution structure of $\mathcal{K}$, we have:
            \begin{align*}
                \frac{\p}{\p \widetilde{z}}\mathcal{K}[\g](\widetilde{r},z)& = \frac{\p}{\p \widetilde{z}}\int_\R \frac{1}{\widetilde{r}}\phi(\frac{\widetilde{z}-t}{\widetilde{r}})\widetilde{\g}(t)dt=\frac{1}{\widetilde{r}}\int_\R \frac{1}{\widetilde{r}}\phi'(\frac{t}{\widetilde{r}})\widetilde{\g}(\widetilde{z}-t) dt,\\
                \frac{\p}{\p \widetilde{r}}\mathcal{K}[\g](\widetilde{r},z)&= -\int_\R \left(\frac{1}{\widetilde{r}^2}\phi(\frac{t}{\widetilde{r}})+\frac{t}{\widetilde{r}^3}\phi'(\frac{t}{\widetilde{r}})\right)\widetilde{\g}(\widetilde{z}-t)dt=-\frac{1}{\widetilde{r}}\int_\R \frac{1}{\widetilde{r}}(t\phi)'(\frac{t}{\widetilde{r}})\g(\widetilde{z}-t)dt.
            \end{align*}
            Note that $\phi',\;(t\phi)'\in \Sw(\R)$. We apply the Sobolev embedding $H^{s-\mu}(\R)\hookrightarrow L^q(\R)$ with $\frac{1}{q}=\frac{1}{2}-(s-\mu)$ if $s-\mu<1/2$ and $q=\infty$ if $s-\mu\geq 1/2$.
            \begin{itemize}
                \item Case $s-\mu <1/2$. Denoting by $q'$ the conjugate exponent of $q$, it follows from H\"older's inequality that
                \begin{align*}
                    \left|\mathcal{K}[\g](\widetilde{r},z)\right| &\leq \widetilde{r}^{-1/q} \|\phi\|_{L^{q'}(\R)} \|\widetilde{\g}\|_{L^q(\R)}=\widetilde{r}^{(s-\mu)-1/2} \|\phi\|_{L^{q'}(\R)} \|\widetilde{\g}\|_{L^q(\R)},\\
                    \left|\frac{\p}{\p \widetilde{z}}\mathcal{K}[\g](\widetilde{r},z)\right| &\leq \widetilde{r}^{(s-\mu)-3/2} \|\phi'\|_{L^{q'}(\R)} \|\widetilde{\g}\|_{L^q(\R)},\\
                    \left|\frac{\p}{\p \widetilde{r}}\mathcal{K}[\g](\widetilde{r},z)\right| &\leq \widetilde{r}^{(s-\mu)-3/2} \|(t\phi)'\|_{L^{q'}(\R)} \|\widetilde{\g}\|_{L^q(\R)}.
                \end{align*}
                \item Case $s-\mu\geq 1/2$. Using Young's inequality, we have
                \begin{align*}
                    |\mathcal{K}[\g](\widetilde{r},z)| &\leq \|\phi\|_{L^1(\R)} \|\widetilde{\g}\|_{L^\infty(\R)},\\
                    \left|\frac{\p}{\p \widetilde{z}}\mathcal{K}[\g](\widetilde{r},z)\right| &\leq \widetilde{r}^{-1} \|\phi'\|_{L^1(\R)} \|\widetilde{\g}\|_{L^\infty(\R)},\\
                    \left|\frac{\p}{\p \widetilde{r}}\mathcal{K}[\g](\widetilde{r},z)\right| &\leq \widetilde{r}^{-1} \|(t\phi)'\|_{L^1(\R)} \|\widetilde{\g}\|_{L^\infty(\R)}.
                \end{align*}
            \end{itemize}
            Define $\widetilde{\mu}:=\mu$ if $s-\mu\geq 1/2$ and $\widetilde{\mu}=s-1/2$ if $s-\mu<1/2$. Substituting these estimates into the expressions for derivatives of $u_{E,\mu}$, we obtain
            \begin{align*}
                \left|\frac{\p}{\p z} u_{E,\mu}(r,\t,z)\right| &\leq \|\widetilde{\g}\|_{L^q(\R)} \left(\frac{\widetilde{r}^{\widetilde{\mu}-1}}{\d(z)} \|\phi'\|_{L^{q'}(\R)} + \frac{|\d'(z)| \widetilde{r}^{\widetilde{\mu}}}{\d(z)}(\|(t\phi)'\|_{L^{q'}(\R)}+\mu \|\phi\|_{L^{q'}(\R)})\right).
            \end{align*}
            The same estimate holds for $\frac{\p}{\p r} u_{E,\mu}$.

            Note that the cut-off function $\eta$ and its support $\ell$ satisfy $\ell \leq 1/2$ and $\eta(\widetilde{r})=0$ if $\widetilde{r}\geq 2\ell$. Hence we may assume $\widetilde{r}\leq 1$. Also, $\d$ is equivalent to the distance to the endpoints of $E$, which implies $\d(z)\geq C_{\d}\rho$ for some constant $C_\d>0$. Then the estimates become
            \begin{align*}
                &\left|\frac{\p}{\p z} u_{E,\mu}(r,\t,z)\right|,\; \left|\frac{\p}{\p r} u_{E,\mu}(r,\t,z)\right|\\
                 &\leq \|\widetilde{\g}\|_{L^q(\R)} C_\d^{-\widetilde{\mu}}\left( \|\phi'\|_{L^{q'}(\R)} + |\d'(z)|(\|(t\phi)'\|_{L^{q'}(\R)}+\mu \|\phi\|_{L^{q'}(\R)})\right)\frac{r^{\widetilde{\mu}-1}}{\rho^{\widetilde{\mu}}},\\
                 & \frac{1}{r} \left|\frac{\p}{\p \t} u_{E,\mu}(r,\t,z)\right| \leq \mu\|\widetilde{\g}\|_{L^q(\R)} C_\d^{-\widetilde{\mu}} \|\phi\|_{L^{q'}(\R)}\frac{r^{\widetilde{\mu}-1}}{\rho^{\widetilde{\mu}}}.
            \end{align*}            
            
            Thus, the claim follows by combining these estimates with \eqref{eq:embedding-weighted}.
        \end{proof}

        \medskip

        \begin{lemma}\label{lemma:K-gamma}
            Let $1/2<\mu$ and let $\mu+1/2<s<\mu+1$. Then, there exists a constant $C'_{\phi,\mu}>0$ such that
            \begin{align*}
               \forall z\in E,\; \widetilde{r}>0,\quad |\mathcal{K}[\g](\widetilde{r},z)-\g(z)|\leq C'_{\phi,\mu} \widetilde{r}^{s-\mu-1/2}\|\g\|_{V_{-s}^{s-\mu}(E)}.
            \end{align*}
        \end{lemma}
        \begin{proof}
            Using the fact $\ds \int_\R \phi(t)dt=1$, we have 
            \begin{align*}
                \mathcal{K}[\g](\widetilde{r},z)-\widetilde{\g}(\widetilde{z})=\int_\R \frac{1}{\widetilde{r}} \phi(\frac{t}{\widetilde{r}})(\widetilde{\g}(\widetilde{z}-t)-\widetilde{\g}(\widetilde{z})) dt.
            \end{align*}
            We split this integral into two parts, one for $|t|\leq \widetilde{r}$ and the other for $|t|>\widetilde{r}$.

            Since $s-\mu>1/2$ and $\widetilde{g}\in H^{s-\mu}(\R)$, it follows from Sobolev embedding that $\widetilde{\g}\in \Cont^{s-\mu-1/2}(\R)$, and for some constant $C>0$ that
            \begin{align*}
                    \forall x,y\in\R,\quad |\widetilde{\g}(x)-\widetilde{\g}(y)| &\leq C \|\widetilde{\g}\|_{H^{s-\mu}(\R)}|x-y|^{s-\mu-1/2}.
            \end{align*}
            Then, we can estimate the first part of the integral as follows,
            \begin{align*}
                \left|\int_{|t|\leq \widetilde{r}} \frac{1}{\widetilde{r}} \phi\left(\frac{t}{\widetilde{r}}\right) (\widetilde{\g}(\widetilde{z}-t)-\widetilde{\g}(\widetilde{z})) dt \right| &\leq C\|\widetilde{\g}\|_{H^{s-\mu}(\R)}\widetilde{r}^{s-\mu-1/2}\|\phi\|_{L^1(\R)}.
            \end{align*}
            For the second part, since $\phi$ is rapidly decreasing, there exists $C'>0$ such that $\ds |x^2\phi(x)| \leq C'$. Then, we can estimate the second part of the integral as follows,
            \begin{align*}
                \left|\int_{|t|>\widetilde{r}} \frac{1}{\widetilde{r}} \phi\left(\frac{t}{\widetilde{r}}\right) (\widetilde{\g}(\widetilde{z}-t)-\widetilde{\g}(\widetilde{z})) dt \right| &\leq 2C'\|\widetilde{\g}\|_{L^\infty(\R)}\int_{\widetilde{r}}^{+\infty} \frac{\widetilde{r}^2}{t^2} dt \leq 2C'\|\widetilde{\g}\|_{L^\infty(\R)}\widetilde{r}.
            \end{align*}
            The claim follows by combining these estimates with the embedding \eqref{eq:embedding-weighted}.
        \end{proof}

    \bibliographystyle{abbrv}
    \bibliography{ref}

@article{Aspri2022,
   author = {Andrea Aspri and Elena Beretta and Elisa Francini and Sergio Vessella},
   doi = {10.1137/22M1480550},
   issn = {0036-1410},
   issue = {5},
   journal = {SIAM Journal on Mathematical Analysis},
   month = {10},
   pages = {5182-5222},
   publisher = {Society for Industrial and Applied Mathematics},
   title = {Lipschitz Stable Determination of Polyhedral Conductivity Inclusions from Local Boundary Measurements},
   volume = {54},
   url = {https://epubs.siam.org/doi/10.1137/22M1480550},
   year = {2022}
}

@article{Choulli2006,
   author = {M. Choulli},
   doi = {10.1080/00036810500499926},
   issn = {1563504X},
   issue = {6-7},
   journal = {Applicable Analysis},
   pages = {693-699},
   publisher = {Taylor \& Francis Group},
   title = {On the determination of an inhomogeneity in an elliptic equation},
   volume = {85},
   url = {https://www.tandfonline.com/doi/pdf/10.1080/00036810500499926},
   year = {2006}
}

@article{Choulli2003,
   author = {M Choulli},
   doi = {10.1088/0266-5611/19/4/307},
   issn = {0266-5611},
   issue = {4},
   journal = {Inverse Problems},
   month = {8},
   pages = {895-907},
   publisher = {IOP Publishing},
   title = {Local stability estimate for an inverse conductivity problem},
   volume = {19},
   url = {https://iopscience.iop.org/article/10.1088/0266-5611/19/4/307},
   year = {2003}
}

@article{Beretta2021,
   author = {Elena Beretta and Elisa Francini and Sergio Vessella},
   doi = {10.1137/20M1369609},
   issn = {0036-1410},
   issue = {4},
   journal = {SIAM Journal on Mathematical Analysis},
   month = {1},
   pages = {4303-4327},
   publisher = {Society for Industrial and Applied Mathematics},
   title = {Lipschitz Stable Determination of Polygonal Conductivity Inclusions in a Two-Dimensional Layered Medium from the Dirichlet-to-Neumann Map},
   volume = {53},
   url = {https://epubs.siam.org/doi/10.1137/20M1369609},
   year = {2021}
}

@article{Alessandrini1988,
   author = {Giovanni Alessandrini},
   doi = {10.1080/00036818808839730},
   issn = {1563504X},
   issue = {1-3},
   journal = {Applicable Analysis},
   pages = {153-172},
   publisher = {Gordon and Breach Science Publishers Ltd.},
   title = {Stable determination of conductivity by boundary measurements},
   volume = {27},
   url = {https://www.tandfonline.com/doi/abs/10.1080/00036818808839730},
   year = {1988}
}

@article{Leitzke2020,
   author = {Juliana Padilha Leitzke and Hubert Zangl},
   doi = {10.3390/S20185160},
   issn = {1424-8220},
   issue = {18},
   journal = {Sensors 2020, Vol. 20, Page 5160},
   month = {9},
   pages = {5160},
   pmid = {32927685},
   publisher = {Multidisciplinary Digital Publishing Institute},
   title = {A Review on Electrical Impedance Tomography Spectroscopy},
   volume = {20},
   url = {https://www.mdpi.com/1424-8220/20/18/5160/htm https://www.mdpi.com/1424-8220/20/18/5160},
   year = {2020}
}

@article{Martins2019,
   author = {Thiago de Castro Martins and André Kubagawa Sato and Fernando Silva de Moura and Erick Dario León Bueno de Camargo and Olavo Luppi Silva and Talles Batista Rattis Santos and Zhanqi Zhao and Knut Möeller and Marcelo Brito Passos Amato and Jennifer L. Mueller and Raul Gonzalez Lima and Marcos de Sales Guerra Tsuzuki},
   doi = {10.1016/J.ARCONTROL.2019.05.002},
   issn = {1367-5788},
   journal = {Annual Reviews in Control},
   month = {1},
   pages = {442-471},
   publisher = {Pergamon},
   title = {A review of electrical impedance tomography in lung applications: Theory and algorithms for absolute images},
   volume = {48},
   url = {https://www.sciencedirect.com/science/article/abs/pii/S1367578819300173},
   year = {2019}
}

@article{Mansouri2021,
   author = {Sofiene Mansouri and Yousef Alharbi and Fatma Haddad and Souhir Chabcoub and Anwar Alshrouf and Amr A Abd-Elghany},
   doi = {10.2478/joeb-2021-0007},
   journal = {J Electr Bioimp},
   pages = {50-62},
   title = {Electrical Impedance tomography-recent applications and developments},
   volume = {12},
   url = {http://creativecommons.org/licenses/by/4.0/},
   year = {2021}
}

@article{Uhlmann2009,
   author = {G. Uhlmann},
   doi = {10.1088/0266-5611/25/12/123011},
   issn = {0266-5611},
   issue = {12},
   journal = {Inverse Problems},
   month = {12},
   pages = {123011},
   publisher = {IOP Publishing},
   title = {Electrical impedance tomography and Calderón's problem},
   volume = {25},
   url = {https://iopscience.iop.org/article/10.1088/0266-5611/25/12/123011 https://iopscience.iop.org/article/10.1088/0266-5611/25/12/123011/meta},
   year = {2009}
}

@article{Costabel1999,
   author = {Martin Costabel and Monique Dauge and Serge Nicaise},
   doi = {10.1051/m2an:1999155},
   issn = {0764-583X},
   issue = {3},
   journal = {ESAIM: Mathematical Modelling and Numerical Analysis},
   month = {5},
   pages = {627-649},
   title = {Singularities of Maxwell interface problems},
   volume = {33},
   url = {http://www.esaim-m2an.org/10.1051/m2an:1999155},
   year = {1999}
}

@article{Costabel1993i,
   author = {Martin Costabel and Monique Dauge},
   doi = {10.1017/S0308210500021272},
   issn = {14737124},
   issue = {1},
   journal = {Proceedings of the Royal Society of Edinburgh Section A: Mathematics},
   pages = {109-155},
   publisher = {Royal Society of Edinburgh Scotland Foundation},
   title = {General edge asymptotics of solutions of second-order elliptic boundary value problems I},
   volume = {123},
   url = {https://www.cambridge.org/core/journals/proceedings-of-the-royal-society-of-edinburgh-section-a-mathematics/article/abs/general-edge-asymptotics-of-solutions-of-secondorder-elliptic-boundary-value-problems-i/54DBE17840E777FA41DEB8CE7F8D4FB2},
   year = {1993}
}

@article{Korevaar1994,
   author = {J. Korevaar and J. L. H. Meyers},
   doi = {10.1112/blms/26.4.353},
   issn = {00246093},
   issue = {4},
   journal = {Bulletin of the London Mathematical Society},
   month = {7},
   pages = {353-362},
   title = {Logarithmic Convexity for Supremum Norms of Harmonic Functions},
   volume = {26},
   url = {http://doi.wiley.com/10.1112/blms/26.4.353},
   year = {1994}
}

@article{Hanke2024,
   author = {Martin Hanke},
   doi = {10.1088/1361-6420/ad76d4},
   issn = {0266-5611},
   journal = {Inverse Problems},
   month = {10},
   publisher = {IOP Publishing},
   title = {Lipschitz stability of an inverse conductivity problem with two Cauchy data pairs},
   year = {2024}
}

@book{Gilbarg-Trudinger,
   author = {David Gilbarg and Neil S. Trudinger},
   city = {Berlin, Heidelberg},
   doi = {10.1007/978-3-642-61798-0},
   isbn = {978-3-540-41160-4},
   publisher = {Springer Berlin Heidelberg},
   title = {Elliptic Partial Differential Equations of Second Order},
   volume = {224},
   url = {http://link.springer.com/10.1007/978-3-642-61798-0},
   year = {2001}
}

@article{Liu2022,
   author = {Hongyu Liu and Chun-Hsiang Tsou},
   doi = {10.1007/s00526-022-02211-w},
   issn = {0944-2669},
   issue = {3},
   journal = {Calculus of Variations and Partial Differential Equations},
   month = {6},
   pages = {91},
   title = {Stable determination by a single measurement, scattering bound and regularity of transmission eigenfunctions},
   volume = {61},
   url = {http://arxiv.org/abs/2108.01557 https://link.springer.com/10.1007/s00526-022-02211-w},
   year = {2022}
}

@article{Friedman1989,
   author = {Avner Friedman and Victor Isakov},
   issue = {3},
   journal = {Indiana University Mathematics Journal},
   pages = {563-579},
   title = {On The Uniqueness in The Inverse Conductivity Problem with One Measurement},
   volume = {38},
   url = {https://www.jstor.org/stable/24895401},
   year = {1989}
}

@article{Cakoni2021,
   author = {Fioralba Cakoni and Jingni Xiao},
   doi = {10.1080/03605302.2020.1843489},
   issn = {0360-5302},
   issue = {3},
   journal = {Communications in Partial Differential Equations},
   month = {3},
   pages = {413-441},
   title = {On corner scattering for operators of divergence form and applications to inverse scattering},
   volume = {46},
   url = {http://arxiv.org/abs/1905.02558 https://www.tandfonline.com/doi/full/10.1080/03605302.2020.1843489},
   year = {2021}
}

@article{Rempel1989,
   author = {Stephan Rempel},
   doi = {10.1007/BF01196880},
   issn = {0378620X},
   issue = {6},
   journal = {Integral Equations and Operator Theory},
   month = {11},
   pages = {835-854},
   publisher = {Birkhäuser-Verlag},
   title = {Corner singularity for transmission problems in three dimensions},
   volume = {12},
   url = {https://link.springer.com/article/10.1007/BF01196880},
   year = {1989}
}

@article{Dauge1999,
   author = {Monique Dauge},
   doi = {10.1051/proc:1999044},
   editor = {Robert Boyer and Elisabeth Croc and Yves Dermenjian and Florence Hubert and Elaine Pratt},
   issn = {1270-900X},
   journal = {ESAIM: Proceedings},
   month = {8},
   pages = {19-40},
   publisher = {EDP Sciences},
   title = {Singularities of Corner Problems and Problems of Corner Singularities},
   volume = {6},
   url = {http://www.esaim-proc.org/10.1051/proc:1999044},
   year = {1999}
}

@article{Calderon1980,
   author = {Alberto P. Calder\'on},
   issue = {2-3},
   journal = {Comput. Appl. Math.},
   pages = {133-138},
   title = {On an inverse boundary value problem},
   volume = {25},
   url = {https://www.scielo.br/scielo.php?pid=S1807-03022006000200002&script=sci_arttext&tlng=es},
   year = {2006}
}

@article{Sylvester1987,
   author = {John Sylvester and Gunther Uhlmann},
   doi = {10.2307/1971291},
   issn = {0003486X},
   issue = {1},
   journal = {The Annals of Mathematics},
   month = {1},
   pages = {153},
   publisher = {JSTOR},
   title = {A Global Uniqueness Theorem for an Inverse Boundary Value Problem},
   volume = {125},
   url = {https://www.jstor.org/stable/1971291},
   year = {1987}
}

@article{Blsten2017,
   author = {Emilia L. K. Blåsten and Hongyu Liu},
   doi = {10.1016/j.jfa.2017.08.023},
   issn = {10960783},
   issue = {11},
   journal = {Journal of Functional Analysis},
   month = {12},
   pages = {3616-3632},
   publisher = {Academic Press Inc.},
   title = {On vanishing near corners of transmission eigenfunctions},
   volume = {273},
   url = {https://linkinghub.elsevier.com/retrieve/pii/S0022123617303373},
   year = {2017}
}

@article{Elschner2015,
   author = {Johannes Elschner and Guanghui Hu},
   doi = {10.1088/0266-5611/31/1/015003},
   issn = {13616420},
   issue = {1},
   journal = {Inverse Problems},
   month = {1},
   publisher = {Institute of Physics Publishing},
   title = {Corners and edges always scatter},
   volume = {31},
   year = {2015}
}

@article{Kondratev1967,
   author = {V. A. Kondrat'ev},
   issn = {01342452},
   journal = {Tr. Mosk. Mat. Obs.},
   month = {10},
   pages = {209-292},
   title = {Boundary value problems for elliptic equations in domains with conical or angular points},
   volume = {16},
   url = {http://www.mathnet.ru/php/archive.phtml?wshow=paper&jrnid=mmo&paperid=186&option_lang=eng},
   year = {1967}
}

@article{Dauge1989Oblique,
   author = {Monique Dauge and Serge. Nicaise},
   doi = {10.1080/03605308908820649},
   issn = {15324133},
   issue = {8-9},
   journal = {Communications in Partial Differential Equations},
   pages = {1147-1192},
   title = {Oblique Derivative and Interface Problems on Polygonal Domains and Networks},
   volume = {14},
   year = {1989}
}

@article{Nicaise1994ii,
   author = {Serge. Nicaise and Anna-Margarete Sändig},
   doi = {10.1002/mma.1670170603},
   issn = {10991476},
   issue = {6},
   journal = {Mathematical Methods in the Applied Sciences},
   pages = {431-450},
   title = {General interface problems—II},
   volume = {17},
   year = {1994}
}

@article{Nicaise1994i,
   author = {Serge. Nicaise and Anna-Margarete Sändig},
   doi = {10.1002/mma.1670170602},
   issn = {10991476},
   issue = {6},
   journal = {Mathematical Methods in the Applied Sciences},
   pages = {395-429},
   title = {General interface problems—I},
   volume = {17},
   year = {1994}
}

@article{Beretta2022,
   author = {Elena Beretta and Elisa Francini},
   doi = {10.1080/00036811.2020.1775819},
   issn = {0003-6811},
   issue = {10},
   journal = {Applicable Analysis},
   month = {7},
   pages = {3536-3549},
   publisher = {Taylor and Francis Ltd.},
   title = {Global Lipschitz stability estimates for polygonal conductivity inclusions from boundary measurements},
   volume = {101},
   url = {https://www.tandfonline.com/doi/full/10.1080/00036811.2020.1775819},
   year = {2022}
}

@article{Kellogg1971,
   author = {R.B. Kellogg},
   doi = {10.1016/b978-0-12-358502-8.50015-3},
   journal = {Numerical Solution of Partial Differential Equations–II},
   month = {1},
   pages = {351-400},
   publisher = {Elsevier},
   title = {Singularities in interface problems},
   year = {1971}
}

@article{Barcelo1994,
   author = {Bartolome Barcelo and Eugene Fabes and Jin-Keun Seo},
   doi = {10.2307/2160858},
   issn = {00029939},
   issue = {1},
   journal = {Proceedings of the American Mathematical Society},
   month = {9},
   pages = {183},
   title = {The Inverse Conductivity Problem with One Measurement: Uniqueness for Convex Polyhedra},
   volume = {122},
   url = {https://www.jstor.org/stable/2160858?origin=crossref},
   year = {1994}
}

@article{Bellout1992,
   author = {Hamid Bellout and Avner Friedman and Victor Isakov},
   doi = {10.2307/2154032},
   issn = {00029947},
   issue = {1},
   journal = {Transactions of the American Mathematical Society},
   month = {7},
   pages = {271-296},
   title = {Stability for an Inverse Problem in Potential Theory},
   volume = {332},
   url = {https://www.jstor.org/stable/2154032?origin=crossref},
   year = {1992}
}

@article{Alessandrini1992,
   author = {Giovanni Alessandrini and Rolando Magnanini},
   issue = {4},
   journal = {Annali della Scuola Normale Superiore di Pisa - Classe di Scienze},
   pages = {567-589},
   title = {The Index of Isolated Critical Points and Solutions of Elliptic Equations in the Plane},
   volume = {Ser. 4, 19},
   url = {http://www.numdam.org/item/ASNSP_1992_4_19_4_567_0},
   year = {1992}
}

@article{Kang1997,
   author = {Hyeonbae Kang and Jin Keun Seo and Dongwoo Sheen},
   doi = {10.1137/S0036141096299375},
   issn = {0036-1410},
   issue = {6},
   journal = {SIAM Journal on Mathematical Analysis},
   month = {11},
   pages = {1389-1405},
   title = {The Inverse Conductivity Problem with One Measurement: Stability and Estimation of Size},
   volume = {28},
   url = {http://epubs.siam.org/doi/10.1137/S0036141096299375},
   year = {1997}
}

@article{Seo1995,
   author = {Jin-Keun Seo},
   doi = {10.1007/s00041-001-4030-7},
   issn = {1069-5869},
   issue = {3},
   journal = {Journal of Fourier Analysis and Applications},
   month = {6},
   pages = {227-235},
   title = {On the Uniqueness in the Inverse Conductivity Problem},
   volume = {2},
   url = {http://link.springer.com/10.1007/s00041-001-4030-7},
   year = {1995}
}

@misc{Kang1999ball,
   author = {Hyeonbae Kang and Jin-Keun Seo},
   doi = {10.1137/S0036139997324595},
   issn = {00361399},
   issue = {5},
   journal = {SIAM Journal on Applied Mathematics},
   pages = {1533-1539},
   title = {Inverse conductivity problem with one measurement: Uniqueness of balls in R3},
   volume = {59},
   year = {1999}
}

@misc{Fabes1999,
   author = {Eugene Fabes and Hyeonbae Kang and Jin-Keun Seo},
   doi = {10.1137/S0036141097324958},
   issn = {00361410},
   issue = {4},
   journal = {SIAM Journal on Mathematical Analysis},
   pages = {699-720},
   title = {Inverse conductivity problem with one measurement: Error estimates and approximate identification for perturbed disks},
   volume = {30},
   year = {1999}
}

@article{Alessandrini1994,
   author = {G. Alessandrini and R. Magnanini},
   doi = {10.1137/S0036141093249080},
   issn = {0036-1410},
   issue = {5},
   journal = {SIAM Journal on Mathematical Analysis},
   month = {9},
   pages = {1259-1268},
   title = {Elliptic Equations in Divergence Form, Geometric Critical Points of Solutions, and Stekloff Eigenfunctions},
   volume = {25},
   url = {http://epubs.siam.org/doi/10.1137/S0036141093249080},
   year = {1994}
}

@book{Dauge1988,
   author = {Monique Dauge},
   city = {Berlin, Heidelberg},
   doi = {10.1007/BFb0086682},
   isbn = {978-3-540-50169-5},
   publisher = {Springer Berlin Heidelberg},
   title = {Elliptic Boundary Value Problems on Corner Domains},
   volume = {1341},
   url = {http://link.springer.com/10.1007/BFb0086682},
   year = {1988}
}

@book{Grisvard1985,
   author = {Pierre Grisvard},
   city = {Boston},
   isbn = {0-273-08647-2},
   publisher = {Pitman Advanced Publishing Program},
   title = {Elliptic Problems in Nonsmooth Domains},
   year = {1985}
}

@article{moi5,
   author = {Hongyu Liu and Chun-Hsiang Tsou},
   doi = {10.1088/1361-6420/ab9d6b},
   issn = {0266-5611},
   issue = {8},
   journal = {Inverse Problems},
   month = {8},
   pages = {085010},
   title = {Stable determination of polygonal inclusions in Calder\'on's problem by a single partial boundary measurement},
   volume = {36},
   url = {http://arxiv.org/abs/1902.04462 https://iopscience.iop.org/article/10.1088/1361-6420/ab9d6b},
   year = {2020}
}

@article{moi6,
   author = {Hongyu Liu and Chun-Hsiang Tsou and Wei Yang},
   doi = {10.1088/1361-6420/abefeb},
   issn = {0266-5611},
   issue = {5},
   journal = {Inverse Problems},
   month = {5},
   pages = {055005},
   title = {On Calderón’s inverse inclusion problem with smooth shapes by a single partial boundary measurement},
   volume = {37},
   url = {http://arxiv.org/abs/2006.10586 https://iopscience.iop.org/article/10.1088/1361-6420/abefeb},
   year = {2021}
}

@article{moi2,
   author = {Faouzi Triki and Chun-Hsiang Tsou},
   doi = {10.1016/j.jde.2020.02.028},
   issn = {10902732},
   issue = {4},
   journal = {Journal of Differential Equations},
   pages = {3259-3281},
   publisher = {Elsevier Inc.},
   title = {Inverse inclusion problem: A stable method to determine disks},
   volume = {269},
   url = {https://doi.org/10.1016/j.jde.2020.02.028},
   year = {2020}
}

@article{Alessandrini2009,
   author = {Giovanni Alessandrini and Luca Rondi and Edi Rosset and Sergio Vessella},
   doi = {10.1088/0266-5611/25/12/123004},
   issn = {0266-5611},
   issue = {12},
   journal = {Inverse Problems},
   month = {12},
   pages = {123004},
   title = {The stability for the Cauchy problem for elliptic equations},
   volume = {25},
   url = {https://iopscience.iop.org/article/10.1088/0266-5611/25/12/123004},
   year = {2009}
}

@article{Blsten2020,
   author = {Emilia L. K. Blåsten and Hongyu Liu},
   doi = {10.1088/1361-6420/ab958f},
   issn = {0266-5611},
   issue = {8},
   journal = {Inverse Problems},
   month = {8},
   pages = {085005},
   title = {Recovering piecewise constant refractive indices by a single far-field pattern},
   volume = {36},
   url = {https://iopscience.iop.org/article/10.1088/1361-6420/ab958f},
   year = {2020}
}

@article{Blsten2014,
   author = {Emilia L. K. Blåsten and Lassi Päivärinta and John Sylvester},
   doi = {10.1007/s00220-014-2030-0},
   issn = {14320916},
   issue = {2},
   journal = {Communications in Mathematical Physics},
   pages = {725-753},
   title = {Corners Always Scatter},
   volume = {331},
   year = {2014}
}

@book{Mazya2010,
   author = {Vladimir Maz'ya and J\"urgen Rossmann},
   city = {Providence, Rhode Island},
   doi = {10.1090/surv/162},
   isbn = {9780821849835},
   month = {4},
   publisher = {American Mathematical Society},
   title = {Elliptic Equations in Polyhedral Domains},
   volume = {162},
   url = {http://www.ams.org/surv/162},
   year = {2010}
}

@article{Blasten2016,
   author = {Emilia L. K. Blåsten and Hongyu Liu},
   doi = {10.1512/iumj.2021.70.8411},
   issn = {0022-2518},
   issue = {3},
   journal = {Indiana University Mathematics Journal},
   pages = {907-947},
   title = {On corners scattering stably and stable shape determination by a single far-field pattern},
   volume = {70},
   url = {http://arxiv.org/abs/1611.03647 http://www.iumj.indiana.edu/IUMJ/fulltext.php?artid=8411&year=2021&volume=70},
   year = {2021}
}

@book{Kozlov2002,
   author = {V. Kozlov and Vladimir Maz'ya and J\"urgen Rossmann},
   city = {Providence, Rhode Island},
   doi = {10.1090/surv/052},
   isbn = {9780821807545},
   month = {11},
   publisher = {American Mathematical Society},
   title = {Elliptic Boundary Value Problems in Domains with Point Singularities},
   volume = {52},
   url = {http://www.ams.org/surv/052},
   year = {2002}
}

@article{Nicaise1999,
   author = {Serge. Nicaise and Anna-Margarete Sändig},
   doi = {10.1142/S0218202599000403},
   issn = {0218-2025},
   issue = {06},
   journal = {Mathematical Models and Methods in Applied Sciences},
   month = {8},
   pages = {855-898},
   title = {Transmission problems for the Laplace and elastivity operators: Regularity and boundary integral formulation},
   volume = {09},
   url = {https://www.worldscientific.com/doi/abs/10.1142/S0218202599000403},
   year = {1999}
}

@article{Nicaise1990,
   author = {Serge. Nicaise},
   doi = {10.1080/03605309908820734},
   isbn = {0360530990},
   issn = {0360-5302},
   issue = {10},
   journal = {Communications in Partial Differential Equations},
   month = {1},
   pages = {1475-1508},
   title = {Polygonal interface problems:higher regularity results},
   volume = {15},
   url = {http://www.tandfonline.com/doi/abs/10.1080/03605309908820734},
   year = {1990}
}

@book{Ammari2008,
   author = {Habib Ammari},
   city = {Berlin, Heidelberg},
   doi = {10.1007/978-3-540-79553-7},
   isbn = {978-3-540-79552-0},
   journal = {An Introduction to Mathematics of Emerging Biomedical Imaging},
   publisher = {Springer Berlin Heidelberg},
   title = {An Introduction to Mathematics of Emerging Biomedical Imaging},
   volume = {62},
   url = {http://link.springer.com/10.1007/978-3-540-79553-7},
   year = {2008}
}
    
\end{document}